\documentclass[12pt]{amsart}
\usepackage{amsthm,amsmath,amsxtra,amscd,amssymb,xypic}
\usepackage[all]{xy}

\newcommand{\ssigma}{{\boldsymbol{\Sigma}}}

\newcommand{\stab}{{\operatorname{stable}}}
\newcommand{\Span}{\operatorname{Span}}
\newcommand{\stratum}[2]{\beta_{#1}({#2})}
\newcommand{\betazero}[1]{\beta^0_{#1}}
\newcommand{\Psing}{\Perf[g,\operatorname{sing}]}
\newcommand{\Psmooth}{\Perf[g,\operatorname{smooth}]}
\newcommand{\Psimp}{\Perf[g,\operatorname{simp}]}
\newcommand{\topd}{\operatorname{top}}

\newcommand{\inv}{\operatorname{inv}}

\newcommand{\CC}{{\mathbb{C}}}

\newcommand{\HH}{{\mathbb{H}}}

\newcommand{\EE}{{\mathbb{E}}}
\newcommand{\PP}{{\mathbb{P}}}
\newcommand{\QQ}{{\mathbb{Q}}}
\newcommand{\RR}{{\mathbb{R}}}
\newcommand{\ZZ}{{\mathbb{Z}}}

\newcommand{\NN}{{\mathbb{N}}}

\newcommand{\VV}{{\mathbb{V}}}

\newcommand{\calT}{{\mathcal T}}
\newcommand{\Torus}{{\mathbb T}}

\newcommand{\calA}{{\mathcal A}}

\newcommand{\calM}{{\mathcal M}}

\newcommand{\calP}{{\mathcal P}}

\newcommand{\calX}{{\mathcal X}}

\newcommand{\calS}{{\mathcal S}}

\newcommand{\op}{\operatorname}
\newcommand{\ab}[1][g]{\calA_{#1}}
\newcommand{\ua}[1][g]{\calX_{#1}}
\newcommand{\Sat}[1][g]{{\calA_{#1}^{\op {Sat}}}}
\newcommand{\Vor}[1][g]{{\calA_{#1}^{\op {Vor}}}}
\newcommand{\Perf}[1][g]{{\calA_{#1}^{\op {Perf}}}}
\newcommand{\Matr}[1][g]{{\calA_{#1}^{\op {Matr}}}}
\newcommand{\Std}[1][g]{{\calA_{#1}^{\op {Std}}}}

\newcommand{\Asigma}[1][g]{{\overline{\calA}_{#1}^\ssigma}}
\newcommand{\Sp}{\op{Sp}}

\newcommand{\GSp}{\op{GSp}}
\newcommand{\GL}{\op{GL}}
\newcommand{\Mat}{\op{Mat}}
\newcommand{\Sym}{\op{Sym}}

\newcommand{\tor}{\op{tor}}

\newcommand\codim{{\rm codim}}
\newcommand\rank{\operatorname{rank}}
\newcommand\ud{\underline}

\newcommand{\pu}{\bullet}
\newcommand{\cohloc}[3]{H^{#1}(#2,#3)}
\newcommand{\coh}[2][\pu]{\cohloc {#1}{#2}{\QQ}}

\newcommand{\cohcloc}[3]{H_c^{#1}(#2,#3)}
\newcommand{\cohc}[2][\pu]{\cohcloc {#1}{#2}{\QQ}}

\theoremstyle{plain}
\newtheorem{thm}{Theorem}[section]
\newtheorem{lm}[thm]{Lemma}
\newtheorem{prop}[thm]{Proposition}
\newtheorem{cor}[thm]{Corollary}

\newtheorem{qu}[thm]{Question}
\newtheorem{conj}[thm]{Conjecture}
\newtheorem{summary}[thm]{Summary}

\theoremstyle{definition}

\newtheorem{rem}[thm]{Remark}

\begin{document}
\title[Stable cohomology of $\overline{\ab}$]{Stable cohomology of the perfect cone toroidal compactification of $\ab$}
\author{Samuel Grushevsky}
\address{Mathematics Department, Stony Brook University,
Stony Brook, NY 11790-3651, USA}
\email{sam@math.sunysb.edu}
\thanks{Research of the first author is supported in part by National Science Foundation under the grant DMS-12-01369.}
\author{Klaus Hulek}
\address{Institut f\"ur Algebraische Geometrie, Leibniz Universit\"at Hannover, Welfengarten 1, 30060 Hannover, Germany.}
\email{hulek@math.uni-hannover.de}
\thanks{Research of the second and third authors is supported in part by DFG grants Hu-337/6-1 and Hu-337/6-2. The final revision of the paper was completed at the
Institute for Advanced Study at Princeton where the second author was supported by the Fund for Mathematics.}
\author{Orsola Tommasi}
\address{Fachbereich Mathematik, Technische Universit\"at Darmstadt, Schlo\ss{}gartenstr. 7, 64289 Darmstadt, Germany}
\email{tommasi@mathematik.tu-darmstadt.de}

\begin{abstract}
We show that the cohomology of the perfect cone (also called first Voronoi) toroidal compactification $\Perf$ of the moduli space of complex principally polarized abelian varieties stabilizes in close to the top degree. Moreover, we show that this stable cohomology is purely algebraic, and we compute
it in degree up to 13. Our explicit computations and stabilization results apply in greater generality to various toroidal compactifications and partial compactifications, and in particular we show that the cohomology of the matroidal partial compactification $\Matr$ stabilizes in fixed degree, and forms a polynomial algebra. For degree up to 8, we describe explicitly the generators of the cohomology, and discuss various approaches to computing all of the stable cohomology in general.
\end{abstract}
\maketitle

\section{Introduction}

The stabilization of cohomology is of great interest in the study of the geometry of moduli spaces.
The most notable results in this direction are the stabilization of the cohomology of the moduli space $\ab$ of $g$-dimensional complex principally polarized abelian varieties (ppav), proved by Borel \cite{borel1}, and of the moduli space $\calM_g$ of non-singular algebraic curves of genus $g$, first proved by Harer in \cite{harerstab}. In both cases, the cohomology with $\QQ$ coefficients is shown to stabilize, in the sense that the degree $k$ cohomology group of the moduli space is independent of $g$ when $g$ is sufficiently large with respect to $k$. In both cases, stable cohomology is freely generated by classes whose geometric meaning is well understood: {}it follows from the work of Borel that the odd $\lambda$-classes generate the stable cohomology $\ab$, while the fact that the $\kappa$-classes generate the stable cohomology of $\calM_g$ is the celebrated theorem of Madsen and Weiss \cite{mawe}, proven using homotopy-theoretic methods.{}

It is natural to wonder whether similar stability occurs also for compactifications of moduli spaces. This is clearly not the case for the Deligne--Mumford compactification $\overline{\calM}_g$ of $\calM_g$, because it is known that the rank of the Picard group of $\overline{\calM}_g$, and hence its second cohomology, grows linearly in $g$.
On the other hand, it was shown by Charney and Lee \cite{chle} that the {}cohomology of the Satake (minimal) compactification $\Sat$ of $\ab$ stabilizes in the same range as $H^k(\ab)$.
For questions in algebraic geometry,
the toroidal compactifications of $\ab$ introduced in \cite{amrtbook}
are most relevant.
The stabilization of cohomology for any toroidal compactification in any range is a completely open problem (see \cite[\S6]{grAgsurvey}), and as we shall see the answer also depends on the compactification chosen{}. Moreover, the question is also interesting for
partial compactifications of $\calA_g$ such as the matroidal partial toroidal compactification.

\smallskip
The main purpose of this paper is to show the stabilization of cohomology in {\em close to the top degree} for the perfect cone toroidal compactification
$\Perf$ of $\ab$. Throughout the
paper we work
with $\QQ$ coefficients, and our main result is the following
\begin{thm}[Main theorem]\label{thm:main}
The cohomology and the homology of the perfect cone compactification stabilize in close to the top degree, i.e.~the groups $H^{g(g+1)-k}(\Perf,\QQ)$ and $H_{g(g+1)-k}(\Perf,\QQ)$ are independent of $g$ for $k<g$.
\end{thm}

Let us recall that the map $\ab\rightarrow \ab [g+1]$ defined by taking the product with a fixed elliptic curve extends to a map $\Perf \rightarrow \Perf[g+1]$ which is a transversal embedding with well-defined normal bundle, after passing to a suitable finite cover. This ensures the existence of Gysin maps $H_{(g+1)(g+2)-k}(\Perf[g+1], \QQ)\rightarrow H_{g(g+1)-k}(\Perf,\QQ)$. In the stable range, these maps induce the stabilization isomorphisms in our theorem.

The method of our proof is by noting that $\Perf$ admits a stratification with strata corresponding to various cones in the perfect cone or first Voronoi decomposition.
First, we prove in Theorem~\ref{thm:stablestrata} that the cohomology of each stratum stabilizes. Then we use the Gysin exact sequence to compute the cohomology of the union of all strata, using the specifics of the perfect fan to argue that the resulting cohomology stabilizes.
This construction can be extended in a straightforward way to homology using long exact sequences in Borel--Moore homology.
In particular, we obtain a stabilization isomorphism $H_{(g+1)(g+2)-k}(\Perf[g+1], \QQ)\rightarrow H_{g(g+1)-k}(\Perf,\QQ)$ in the stable range which restricts to the usual Gysin map on each toroidal stratum.

If one considers the cycle map to homology on the singular space $\Perf$, the constructions above allow us to see where the stable homology
classes come from, proving the next result:
\begin{thm}\label{algebraic}
The stable homology groups $H_{g(g+1)-k}(\Perf,\QQ)$
for $k<g$ are generated by algebraic classes.
\end{thm}

If Poincar\'e duality were to hold, one could relate the close to top degree cohomology groups $H^{g(g+1)-k}(\Perf,\QQ)$  to $H^{k}(\Perf,\QQ)$. Since, however,
the perfect cone toroidal compactification is singular, there is no a priori reason for Poincar\'e duality
to hold. Indeed, our computations in  genus $4$ \cite{huto2} show that Poincar\'e duality does fail, although these computations still allow for the possibility for Poincar\'e duality to
hold in the stable range.
A different approach would be to look at the intersection cohomology of $\Perf$.
It was recently shown by Dutour Sikiri\'c, Sch\"urmann, and the second author in \cite{DutourHulekSchuermann}, that for $g\geq4$ the locus of singular points of the stack $\Perf$ has codimension $10$ in $\Perf$ (while $\Perf$ is smooth as a stack for $g\le 3$).
This implies by \cite[Prop.~3]{durfeebetti} that $IH^k(\Perf,\QQ)=H^{g(g+1)-k}(\Perf,\QQ)$ for $k\leq 10$ for the middle perversity intersection cohomology of $\Perf$.
Moreover, by the results in \cite{bbfgk} algebraic cycles can always be lifted to intersection homology. Combining this with the two theorems above, we get that the stable homology $H_{g(g+1)-k}(\Perf,\QQ)$ can be lifted to $IH_{g(g+1)-k}(\Perf,\QQ)\cong IH^k(\Perf,\QQ)$.
This motivates the following

\begin{qu}\label{question_stableih}
Does the intersection cohomology of the perfect cone compactification stabilize, more specifically, is it true that the homomorphism $IH^{k}(\Perf, \QQ)\twoheadrightarrow H_{g(g+1)-k}(\Perf,\QQ)$ is an isomorphism for all $k<g$?
\end{qu}

As the stability map $\Perf\rightarrow \Perf[g+1]$ is (in an orbifold sense) a transversal embedding of pure dimension, there is a well-defined map for 
intersection cohomology  $IH^{k}(\Perf[g+1],\QQ)\rightarrow IH^{k}(\Perf[g],\QQ)$.
Combining this with the (hypothetical) isomorphism from Question~\ref{question_stableih} would prove that also the intersection cohomology of $\Perf$ stabilizes in the range $k<g$.

It is, at this stage, opportune to go briefly back and consider the situation for the Satake compactification $\Sat$.
Recall that the stable cohomology of the Satake compactification was
computed by Charney and Lee:
\begin{thm}[\cite{chle}]
For $k$ fixed and $g>k$ the rational cohomology $H^{k}(\Sat,\QQ)$ does not depend on $g$, and the stable
cohomology ring is freely generated by  classes $\lambda_1,\lambda_3,\lambda_5,\ldots$,
and  $\alpha_3,\alpha_5,\alpha_7,\ldots$ where both $\lambda_i$ and $\alpha_i$ are in degree $2i$.
\end{thm}
Here the $\lambda$-classes are extensions of the Chern classes $\lambda_i=c_i(\EE) \in H^{2i}(\ab, \QQ)$ of the Hodge bundle $\EE$.
The Hodge bundle does not extend to $\Sat$ but by  \cite{mumhirz}, \cite[\S V.2]{fachbook} it extends to
any toroidal compactification and the pullback of the classes $\lambda_i$  on $\Sat$ to a smooth projective toroidal compactification are  the Chern classes of the extended Hodge bundle.

The geometric meaning of the $\alpha_i\in H^{2i}(\Sat,\QQ)$ is less clear. 
By the results of \cite{hain} there is a non-algebraic class (it has a wrong Tate twist)
in $H^6(\Sat[3])$, which is likely to be $\alpha_3$,
and it follows from the results of \cite{huto2} that there is also a non-algebraic class in $H^8(\Sat[4])$, which is likely to be $\alpha_3\lambda_1$.
Furthermore, Chen and Looijenga \cite{chen-looijenga} recently proved that all $\alpha_i$ are of Hodge type $(0,0)$, which in particular implies that they are not algebraic.

On the other hand using our methods it is easy to see that the cohomology  of $\Sat$ in close to top degree also stabilizes, and we can compute it explicitly:
\begin{thm}\label{thm:Satake}
The cohomology $H^{g(g+1)-k}(\Sat,\QQ)$ is independent of $g$ for $k<g$,
and is dual to the truncated free algebra generated by the odd Hodge classes $\lambda_{2i+1}$.
\end{thm}
We recall that the starting point for the study of the stable cohomology of $\calA_g$ is the theorem of Borel (see Theorem~\ref{thm:stable} below), which says that the cohomology $H^k(\calA_g)$ for $k<g$ is freely generated by the classes $\lambda_1,\lambda_3,\lambda_5,\ldots$. Thus the theorem above says that the stable cohomology in close to top degree of $\Sat$ is dual to this, which is expected to be the algebraic part of the stable cohomology of $\Sat$.

We shall now return to toroidal compactifications and partial compactifications, in particular the perfect cone compactification $\Perf$. As we explained, the principal ingredient of our method is that we  prove the stabilization for each of the toroidal strata, using representation theory, and then by assembling this information  using the Gysin spectral sequence. As at each step we are doing explicit manipulations, as a result we get an {\em effective} procedure to compute the dimensions of the stable cohomology groups (and also to say something about their generators). While this quickly becomes very involved combinatorially,
for low degree we get the following result:

\begin{thm}\label{thm:Perfnumbers}
The stable Betti numbers of the perfect cone compactification (i.e.~$\dim_\QQ H^{g(g+1)-k}(\Perf,\QQ)$ for $k < g$) in even degree are as follows:
$$
\begin{array}{|r|rrrrrrr|}
\hline
k&0&2&4&6&8&10&12\\
\hline
&&&&&&&
\\[-2.2ex]
\dim_\QQ H^{g(g+1)-k}(\Perf,\QQ)&1&2&4&9&18&38&83\\[0.3ex]
\hline
\end{array}
$$
Moreover, the stable cohomology $H^{g(g+1)-k}(\Perf,\QQ)$ vanishes for odd $k\le13$.
\end{thm}
\begin{rem}
We note that similar questions are also currently under investigation by Jeffrey Giansiracusa and Gregory Sankaran \cite{gisa}. Their techniques are mostly topological, and at the moment it appears that their method would yield the stabilization of the cohomology of the matroidal locus in low degree with $\ZZ[1/2]$ coefficients, i.e.~the independence of $H^k(\Matr,\ZZ[1/2])$ of $g$ for $g\gg k$. It does not at the moment appear that their method would yield a way to explicitly identify the generators or compute the dimensions of stable cohomology, and thus {}
their results are in a sense rather complementary to ours.
\end{rem}
In fact, our technique also applies to show that the cohomology of the matroidal locus stabilizes. We recall that the matroidal locus $\Matr$ is a partial toroidal compactification of $\ab$ obtained by taking the union of strata corresponding to all matroidal cones.
Melo and Viviani \cite{mevi} showed that a cone is contained in both the perfect cone decomposition and the second Voronoi decomposition if and only if it is a matroidal cone.
In particular the matroidal locus is the biggest partial toroidal compactification  contained in both $\Perf$ and $\Vor$ as a Zariski open subset.
Thus the results of Alexeev and Brunyate \cite{albr} imply that the Torelli map $\calM_g\to\calA_g$ extends to a morphism $\overline{\calM_g}\to\Matr$ from the Deligne--Mumford compactification.
Our  result for the matroidal locus is the following:
\begin{thm}\label{thm:Matrstabilizes}
The cohomology  of the matroidal partial toroidal compactification stabilizes, i.e.~$H^k(\Matr,\QQ)$ does not depend on $g$ for $k<g$. The stable cohomology is generated by algebraic classes.
\end{thm}
{}Since all matroidal cones are simplicial by \cite[Theorem 4.1]{errydicing} and thus define rationally smooth toric varieties, the coarse moduli space of $\Matr$ is a rational homology manifold.{}
Hence the stability result for cohomology in Theorem \ref{thm:Matrstabilizes} is equivalent to a stability result for cohomology with compact support in close to the top degree.

Similarly to the above, as a corollary of our work for the perfect cone compactification we obtain (essentially simply by omitting all the
non-matroidal strata) the {}dimensions of{} stable cohomology of the matroidal locus in degree up to 12:
\begin{thm}\label{thm:Matrnumbers}
The stable cohomology of the matroidal locus in low degree vanishes in odd degree and is given by the following table in even degree
$$\begin{array}{|r|rrrrrrr|}
\hline
k&0&2&4&6&8&10&12\\
\hline
&&&&&&&
\\[-2.2ex]
\dim H^k(\Matr,\QQ)&1&2&4&9&18&37&78\\[0.3ex]
\hline
\end{array}$$
\end{thm}

Proving stabilization in low degree has the advantage that in this case, stable cohomology has a natural structure of a graded Hopf algebra.
As was pointed out to us by Nicholas Shepherd-Barron, when the cohomology of a (partial) compactification $\overline{\ab}$ of $\ab$ stabilizes, the stable cohomology can be identified with the cohomology of the inductive limit $\overline{\ab[\infty]}$ of the sequence of maps $\cdots\rightarrow\overline{\ab}\rightarrow\overline{\ab[g+1]}\rightarrow\cdots$. In particular, whenever the map $\ab[g_1]\times\ab[g_2]\rightarrow\ab[g_1+g_2]$ defined by taking the product of abelian varieties extends to a map of compactifications
$\overline{\ab[g_1]}\times\overline{\ab[g_2]}\rightarrow\overline{\ab[g_1+g_2]}$ for all $g_1,g_2\geq 0$, the inductive limit $\overline{\ab[\infty]}$ has a natural structure as an H-space. Then one can apply Hopf's theorem (see e.g. \cite[Thm. 3C.4]{hatcher}) to conclude that the rational cohomology of $\overline{\ab[\infty]}$, i.e.~the stable cohomology of $\overline{\ab}$, is a free graded-commutative algebra, the tensor product of an exterior algebra on odd-degree generators and a polynomial algebra on even-degree generators.{}

In particular, in the case of $\Matr$, the two theorems above,  together with  some results from Section~\ref{sec:alg}, imply the following:
\begin{cor}\label{cor:matr}
The stable cohomology of $\Matr$ is a polynomial algebra generated by algebraic classes. In low degree $k\le 12$, a possible choice of generators is given by $\lambda_1$, $\lambda_3$, $\lambda_5$, the fundamental classes of the strata of $\Matr$ corresponding to the matroidal cones of dimension smaller than or equal to $6$, one additional generator in degree $10$, and two in degree $12$.
\end{cor}

In Section \ref{sec:alg} we will also discuss two natural subrings of the cohomology ring of the partial compactification consisting of all the simplicial (i.e.~corresponding to stack-smooth strata) cones of $\Perf$, which includes the smooth locus and the matroidal locus  $\Matr$ as open subsets.
More precisely we will investigate
what we call the {\em strata algebra}, which is generated by the fundamental classes of the strata corresponding to the various simplicial cones in the decomposition,
and the {\em  boundary algebra}, which is generated by polynomials in divisorial boundary components of $\Perf(2)$ invariant under the deck group action, where in each case we also add
the Hodge classes $\lambda_i$. This gives a supply of geometrically defined  cohomology classes. It turns out that neither of these subrings suffices to generate the entire
cohomology.
\begin{prop}
Neither the strata algebra nor the boundary algebra  span
the stable cohomology of  $\Perf$.
\end{prop}
This proposition is simply a numerical statement --- at the end of the paper we will see that both
the boundary algebra and the strata algebra in degree 12 have dimension less than the stable cohomology.

Indeed, one expects the geometrical interpretation of stable cohomology to be easier when restricting to suitable open subsets of $\Perf$,  as also the
discussion of the matroidal locus $\Matr$ shows. One can ask this question not only for the matroidal locus, but also
for other geometrically relevant partial compactifications.
A first step in this direction is Theorem~\ref{thm:Stdstabilizes}, where we prove that the stable cohomology of the union $\Std$ of the strata of $\Perf$ corresponding to the standard cones $\langle x_1^2,\dots,x_n^2\rangle$ is freely generated over the stable cohomology of $\ab$ by the fundamental classes of the strata.

\section*{Acknowledgements}
The first author is grateful to Dmitry Zakharov for many stimulating discussions on topics related to the cohomology of families of abelian varieties. We are grateful to Ben Moonen for bringing results of Thompson \cite{thompson} to our attention, and indicating how they relate to the questions we address in this paper. We thank Jayce Getz who asked us in a talk whether a statement such as Theorem \ref{algebraic} could hold. We thank Nicholas Shepherd-Barron and Giulio Codogni for comments and explanations regarding stable cohomology having the Hopf algebra structure.
We are indebted to Eduard Looijenga for useful discussions and suggestions, and Mathieu Dutour Sikiri\'c for answering our questions on the combinatorics of cone decompositions.
The second author would like to thank Barbara Fantechi for discussions on stacks and their singularities.
Finally, the third author would like to thank Dick Hain for answering her questions about the stable cohomology of non-trivial local systems over $\ab$.
We all would like to thank the referees for their careful reading of the paper. The final revision of the paper was completed at the Institute for Advanced Study in Princeton and
the first and the second author would like to thank IAS for excellent working conditions.

\section{Method of proof}

In this section, we review the ideas and the techniques involved in computing cohomology of toroidal compactifications.

We start by pointing out that working in the stable range has the powerful advantage that one can make use of Borel's results on the stable cohomology of the group $\Sp(2g,\ZZ)$.
Indeed, the stable cohomology of $\ab$, being a $K(\Sp(2g,\ZZ),1)$, is equal to the stable cohomology of $\Sp(2g,\ZZ)$, and was computed by Borel \cite{borel1, borel2}.
Moreover, it is also known that the stable cohomology of all non-trivial irreducible rational
local systems on $\ab$ (equivalently, of irreducible rational representations of the algebraic group $\Sp(2g)$) is simply zero
by a strengthening \cite[Theorem 3.2]{hain-infinitesimal} of Borel's stability theorem.

We now proceed to compute the cohomology of various partial toroidal compactifications of $\ab$ obtained by adding various boundary strata.
To explain this,
we first recall that any toroidal compactification $\ab^{\tor}$ of $\ab$ admits a natural map $\varphi:\ab^{\tor} \to \Sat$
to the Satake compactification. The latter is the disjoint union
\begin{equation}\label{equ:beta}
  \Sat = \ab \sqcup \ab[g-1] \sqcup \ldots \sqcup \ab[0],
\end{equation}
and we set $\beta_i:=\varphi^{-1}(\Sat[g-i])$ and
$\beta_i^0:= \beta_i \setminus \beta_{i-1}=\varphi^{-1}(\ab[g-i])$. Each $\beta_i^0$ in turn is stratified
by sets $\stratum{}\sigma$ where $\sigma$ runs through all cones in the perfect cone decomposition
of $\Sym^2_{\geq 0}(\RR^{i})$ whose general element has rank $i$.   We shall refer to such cones as {\it rank $i$ cones}.
The stratum $\stratum{}\sigma$ is the quotient
of a torus bundle $\calT(\sigma)$ over the $i$-fold fiber product
$\ua[g-i]^{\times i}:=\ua[g-i]\times_{\ab[g-i]}\ldots\times_{\ab[g-i]}\ua[g-i]\to\ab[g-i]$ of the universal family by a finite group $G(\sigma)$, namely
the stabilizer
of the cone $\sigma$ in $\GL(i,\ZZ)$. The codimension of  $\stratum{}\sigma$ in $\ab^{\tor} $ equals the dimension
of $\sigma$.

Throughout this work we shall make use of the perfect cone toroidal compactification $\Perf$.
Our method of computing the stable cohomology of $\Perf$, and the outline of the paper, are as follows. In section
\ref{sec:Sp} we recall the relevant results of Borel and Hain, and other necessary background on representations
of the symplectic group. At the end of that section, as a warmup, we prove by using Gysin's exact sequence
the stabilization (in close to top degree) of the cohomology of the Satake compactification, proving Theorem
\ref{thm:Satake}.

For the toroidal case, in Section \ref{sec:Xg} we compute the stable cohomology of the universal family  $\ua$.
This result will be  generalized later in  Section \ref{sec:Xgn} to a computation of the stable cohomology of the fiber product $\ua[g]^{\times n}$ for $n$ fixed.
Since Mumford's partial toroidal compactification $\ab\sqcup\beta_1^0$ is equal to $\ab\sqcup(\ua[g-1]/\i)$,
by using the Gysin exact sequence, in Section \ref{sec:Agpartial} we are then able to compute the stable cohomology of the partial toroidal compactification. Note that these computations are in fact easier in the stable cohomology
than similar calculations in \cite{huto,huto2} for $g=3,4$, as it turns out that in the stable range only the even degree part of cohomology is non-zero, and thus the Gysin long exact sequence breaks up into short exact sequences.

This idea --- of computing the cohomology of an individual stratum,  and then gluing it to the union of the previously considered strata --- is the method that allows us to prove the existence of the stable cohomology of the perfect cone compactification (and by restriction --- of the matroidal locus) in general.
In Section \ref{sec:basic} we review the construction of the perfect cone toroidal compactification, and prove its various combinatorial properties.
In Section \ref{sec:stabletorusbundles} we then use the  Leray spectral sequence to argue that the cohomology
of the torus bundles $\calT(\sigma)$, with $\sigma$ fixed, and $g$ varying, stabilizes, and moreover (only in the stable range!) vanishes in odd degree.

This computation requires dealing with certain local systems $\VV_{\ud{\mu}}$ corresponding to
irreducible representations
of the algebraic group $\Sp(2g)$
indexed by some partition $\ud{\mu}$. By the results of Borel and Hain we know the stable cohomology
$H^k(\ab,\VV_{\ud{\mu}})$, for $k<g$, of which we have to compute the $G(\sigma)$-invariant part. We finally compute the stable cohomology (still in low degree, as opposed to the close to top degree) of each individual stratum $\stratum{}\sigma$, and note that this computation in fact works for any cone, not necessarily just a cone in the perfect cone decomposition.

Adding the strata one by one, we use the Gysin (excision) exact  sequence for cohomology with compact support
to compute the stable cohomology of various partial toroidal compactifications of $\ab$.
For this, we need both the new stratum itself, and the total space obtained to be smooth, so that we can
use Poincar\'e duality to identify $H^{\operatorname{top}-k}_c$ with the previously computed $H^k$ (for $k<g$). In Section \ref{sec:stable} we argue that this process indeed stabilizes,
thus proving our main Theorem \ref{thm:main} on the stabilization of the cohomology of $\Perf$ in close to top degree; our proof also yields the stabilization of the cohomology of $\Matr$ in low degree, Theorem \ref{thm:Matrstabilizes}.
Finally, we use the same techniques to describe the stable cohomology of the open subset of $\Matr$ consisting of degenerations given by standard cones and prove that its stable cohomology is freely generated by the odd $\lambda$-classes and the fundamental classes of the boundary strata (Theorem \ref{thm:Stdstabilizes}).

In Section \ref{sec:next} we demonstrate how this process works, by dealing explicitly with $\beta_2^0$, which is the union of two strata corresponding to semiabelic varieties with the normalization of the toric part being $\PP^1\times \PP^1$ and two copies of $\PP^2$, respectively. Moreover,
applying methodically the procedure described above yields in particular $\coh[\topd-k]{\Perf}$
for $k\le 13$, and we give the results of these computations in Section \ref{sec:numbers}, proving Theorem \ref{thm:Perfnumbers}, from
which Theorem \ref{thm:Matrnumbers} easily follows.
We note that computing the entire stable cohomology of $\Perf\setminus\Perf[g,\operatorname{sing}]$ or of the matroidal locus by our method currently seems
out of reach, as it would involve going through all the possible combinatorics of the strata. Finally in Section \ref{sec:alg} we construct algebraic representatives for (much of) the stable cohomology in low degree. Moreover we discuss the strata algebra and the boundary algebra.

\section{Review of  stable cohomology of local systems on $\ab$}\label{sec:Sp}
Symplectic local systems over $\ab$ play a central role in our computations. Let us start by fixing the notation.
Let $\ud\mu = (\mu_1\geq \mu_2\geq\dots\geq \mu_g)$ be a Young diagram with at most $g$ rows. Equivalently, we can view $\ud\mu$ as an arbitrary partition of length at most $g$. If we denote by $\VV$ the standard rational representation of the group scheme $\Sp(2g)$, then the representation $\VV_{\ud\mu}$ of $\Sp(2g)$ is the irreducible representation of highest weight in the tensor product
$$
\Sym^{\mu_1-\mu_2}(\VV)\otimes \Sym^{\mu_2-\mu_3}\big(\bigwedge^2\VV\big) \otimes \dots\otimes \Sym^{\mu_{g-1}-\mu_g}\big(\bigwedge^{g-1}\VV\big)\otimes\big(\bigwedge^{g}\VV\big)^{\otimes\mu_g}.
$$

One can generalize the definition of $\VV_{\ud\mu}$ to obtain a local system over $\ab$, by applying the same construction as above, but now
setting $\VV$ to be the local system $R^1\pi_*\QQ$, where  $\pi:\;\ua\rightarrow\ab$ is the universal family over $\ab$. Note that $\VV_{\ud\mu}$, defined in this way, is naturally a Hodge module of weight equal to the weight $w(\ud\mu)=\sum_{i=1}^{g}\mu_i$ of $\ud\mu$.
 One can obtain more Hodge modules by taking Tate twists of $\VV_{\ud\mu}$. We will denote such Tate twists by
$$
\VV_{\ud\mu}(k)=\VV_{\ud\mu}\otimes \QQ(k)
$$
for all Young diagrams $\ud\mu$ with at most $g$ rows and all $k\in\ZZ$; the Hodge weight of $\VV_{\ud\mu}(k)$ is then $w(\ud\mu)-2k$. Such Tate twists can be  interpreted in the context of representation theory by working with representations of the group of symplectic similitudes
$$\GSp(2g,\QQ)=\left\{M\in\Mat(2g,2g)|MJ{}^tM=\eta J, \eta\in\QQ^*\right\},$$
where $J=\left(\begin{smallmatrix}0&{\bf 1}_g\\{-\bf 1}_g&0\end{smallmatrix}\right)$ denotes the symplectic matrix.
The Tate Hodge module $\QQ(-1)$ is the inverse $\eta^{-1}$ of the multiplier representation $\eta:\;\GSp(2g,\QQ)\rightarrow\QQ^*$, and $\VV$ is the product of the standard representation of $\GSp(2g,\QQ)$ and $\eta^{-1}$.

\begin{thm}[\cite{borel1,borel2}, {\cite[Theorem 3.2]{hain-infinitesimal}}]
\label{thm:stable}
For the group cohomology of the symplectic group with coefficients in the rational representation $\VV_{\ud\mu}$,
for all $k<g$ we have
$$
 H^k(\Sp(2g,\ZZ),\VV_{\ud{\mu}})=\begin{cases}
 \QQ[x_2,x_6,x_{10}, \ldots]_k& \hbox{if $\ud{\mu}=0$}\\
 0& \hbox{otherwise,}\\
  \end{cases}
$$
where in the first case
 this is the degree $k$ subspace of the graded ring generated by classes $x_i$. In particular the stable cohomology is zero in every odd degree.
\end{thm}
The classes $x_i$ in fact are algebraic and have a geometric meaning; the geometric content of the above theorem is the following
\begin{cor}[Borel \cite{borel1,borel2}]\label{Ag}
The stable cohomology of $\ab$ is freely generated by the Chern classes $\lambda_{2i+1}:=c_{2i+1}(\EE)$ of the Hodge bundle, i.e.~for $k<g$ the vector space $H^{k}(\ab,\QQ)$ is the vector space generated by monomials of total degree $k$ in $\lambda_1,\lambda_3,\ldots$ (where the degree of $\lambda_{2i+1}$ is equal to $4i+2$).
\end{cor}

As a warmup, and to show one little step of our general machinery, we now prove Theorem \ref{thm:Satake}
on the stable cohomology of the Satake compactification, by using the Gysin sequence (to be discussed
in more detail below).
\begin{proof}[Proof of Theorem \ref{thm:Satake}]
Indeed, recall that the Satake compactification $\Sat$ is the union $\ab\sqcup\ab[g-1]\sqcup\ldots\sqcup\ab[0]$, i.e.~we
have $\Sat=\ab\sqcup\partial\Sat=\ab\sqcup\Sat[g-1]$. Thus by the Gysin exact sequence for a closed subvariety
$\Sat[g-1]\subset\Sat$
(see \cite[Cor. 5.51]{PetersSteenbrink})
\begin{multline}
 \ldots\to \coh[\ell-1]{\Sat[g-1]}\to
\cohc[\ell]{\ab}
\to \coh[\ell]{\Sat}
\to \\
\to \coh[\ell]{\Sat[g-1]}
\to \cohc[\ell+1]{\ab}\to\ldots
\end{multline}
In the stable range $g>k$ we have $\ell=g(g+1)-k>g(g-1)=2\dim_\CC\Sat[g-1]$, hence the cohomology of $\Sat[g-1]$ vanishes, so that we simply get $\cohc[g(g+1)-k]\Sat=\cohc[g(g+1)-k]\ab$ for $g>k$.
The latter cohomology by the Poincar\'e duality for the smooth (stack) $\ab$ is dual to $H^k(\ab)$, which equals $\QQ[\lambda_1,\dots,\lambda_{2m+1},\dots]_k$ by Borel's stability theorem (see corollary \ref{Ag}). This implies our claim.
\end{proof}
\begin{rem}
We note that if trying to compute the stable cohomology of the image of the Deligne--Mumford compactification $\overline{\calM_g}$ of the moduli space of curves in $\Vor,\Perf,$ or $\Matr$, to which the Torelli maps extends by \cite{namikawabook},\cite{albr},\cite{mevi}, respectively, we would fail, i.e.~the cohomology would not stabilize. Indeed, for any $i>1$ the Torelli map on the boundary divisor $\Delta_i\subset\overline{\calM_g}$ would extend to an embedding of $\calM_i\times\calM_{g-i}$ into the Satake compactification of $\calM_g$, and then to any compactification of $\calA_g$, so that we would have $\lfloor g/2\rfloor-2$ loci in the Torelli image of $\overline{\calM_g}$ each of codimension 3, which similarly to above would each contribute a class to the stable cohomology in degree $\topd-6$.
\end{rem}

Our proof of stabilization theorems \ref{thm:main} and \ref{thm:Matrstabilizes} for $\Perf$
and $\Matr$ will also use the Gysin exact sequence to compute the cohomology step-by-step by gluing the strata
together. However, notice that the situation will be much more involved, as in both of these cases the boundary
has complex codimension 1, and its cohomology will play a role in the computation. We will thus need to
understand the stable cohomology of individual boundary strata, and will start by investigating the first one,
the boundary of the partial toroidal compactification, which is the universal Kummer family.

\section{Leray spectral sequence and the stable cohomology of the universal family of ppav $\ua\to\ab$}\label{sec:Xg}
In this section we use the Leray spectral sequence to set up the computation of the stable cohomology of a fixed stratum in a toroidal compactification, and demonstrate how this method works by computing the stable cohomology of the universal family of ppav.
 
Indeed, let $\pi:\ua\to\ab$ be the universal family of ppav, considered as a stack. In particular all fibers of $\pi$ are abelian varieties, whereas the generic fiber of the associated map on coarse moduli spaces is actually the Kummer variety, as any ppav has the involution $\i: z \mapsto -z$.

The Leray spectral sequence computes the cohomology of the universal family $H^\pu(\ua,\QQ)$ in terms of local systems on the base. Indeed, it has terms of the form $E_2^{p,q}:=H^p(\ab,R^q\pi_*\QQ)$, and converges $E^{p,q}_\pu\Rightarrow H^{p+q}(\ua,\QQ)$.
To understand the higher direct images under $\pi_*$, recall that the cohomology of an abelian variety is the exterior algebra over the
space of  one-forms, i.e.~$\coh{A}=\bigwedge^\pu H^1(A,\QQ)$.
This description globalizes to describe the higher direct images of the constant sheaf $\QQ$ on $\ua$. Since globally the first cohomology gives the local system $\VV_1$ on $\ab$ (corresponding to the standard representation of $\Sp(2g)$ on $\QQ^{2g}$), we need to recall the formula for the decomposition of exterior powers of the standard representation of the symplectic group into a sum of irreducible representations, which by \cite[Theorem~17.5]{fuhabook} is
\begin{equation}\label{wedgedecompose}
\bigwedge^i{\VV_1}=\bigoplus_{j=0}^{\lfloor i/2\rfloor}\VV_{1^{i-2j}}(-j)
\end{equation}
for $i\leq g$.
We thus obtain
\begin{lm}\label{Rqpi}
For the universal family $\pi:\ua\to\ab$, for any $q\leq g$
$$
 R^q\pi_*\QQ=
 \bigwedge^{q} \VV_1=\VV_{1^{q}}(0)\oplus\VV_{1^{q-2}}(-1)\oplus\cdots\oplus\VV_{p(q)}(-\lfloor q  /2\rfloor),
$$
with $p(q)$ being the remainder of $q$ modulo $2$.
\end{lm}
\begin{rem}
The cohomology of local systems of odd weight over $\ab$ vanishes in odd degrees. In particular, this means that the cohomology of $R^q\pi_*\QQ$ vanishes if $q$ is odd. This reflects the fact that the cohomology of $\ua$ (recall that we always work with rational coefficients) coincides with the cohomology of its coarse moduli space, the universal Kummer family, and that the odd degree cohomology of the Kummer variety $A/\i$ is $0$ for every abelian variety $A$ because it is simply the subspace of the cohomology of the ppav $A$ that is invariant under the involution $\i:\;z\mapsto -z$.
\end{rem}
Using the Leray spectral sequence now allows us to compute the stable cohomology of the universal family:
\begin{prop}\label{Xg}
The stable cohomology of the universal family $\ua$ is generated over the stable cohomology of $\ab$ by the class $\Theta$ of the universal theta divisor trivialized along the zero section. More precisely, this means that for any $k<g$, the vector space $H^k(\ua)$ is  generated by degree $k$ monomials in the classes $\Theta,\lambda_1,\lambda_3,\ldots$, where $\Theta$ has degree 2 and $\lambda_{2i+1}$ has degree $4i+2$. In particular the stable cohomology in any odd degree is zero.
\end{prop}
\begin{proof}
Indeed, combining Lemma \ref{Rqpi} above with Theorem \ref{thm:stable} on the stable cohomology of local systems on $\ab$, we get for $p<g$
$$
 E_2^{p,q}=H^p(\ab,R^q\pi_*\QQ)=\begin{cases}
  H^p(\ab,\QQ)& \hbox{if $p$ and $q$ are even}\\ 0 & \hbox{else}.
  \end{cases}
$$
From a theorem of Deligne \cite{deligne} it follows that the Leray spectral sequence for the projective map $\pi$ degenerates at $E_2$. In our case this is in fact immediate to see directly: in the stable range $p+q\le g$ only the terms of $E_2$ with both $p$ and $q$ even are  non-zero, thus for any differential
$$
 d_r:E_r^{p,q}\to E_r^{p+r,q-r+1}
$$
either the source or the target space is zero, and therefore all differentials vanish. We thus obtain
$$
 H^k(\ua,\QQ)=\bigoplus\limits_{p+q=2k} E_\infty^{p,q}=\bigoplus\limits_{p+q=2k} E_2^{p,q}=\bigoplus\limits_{i=0}^k H^{2i}(\ab,\QQ)(i-k),
$$
for $k=2j<g$ even, while $H^k(\ua,\QQ)=0$ for $k<g$ odd.
(Here we have a direct sum of Hodge structures because for $p+q=k<g$ all $E_2^{p,q}$ carry Tate Hodge-structures of the same weight.)
In words, the above statement says that the stable cohomology of $\ua$ in degree $2j$ is the sum of stable cohomology of $\ab$ in all even degrees up to $2j$, i.e.~for each $i\le j$ we have a copy of $H^{2i}(\ab)$. This means that as an algebra over the stable cohomology of $\ab$, the stable cohomology of $\ua$ is generated by one element, of degree 2 and Hodge type $(1,1)$.
Indeed, denote by $\Theta\subset\ua$ the universal symmetric theta divisor trivialized along the zero section. Then under the decomposition $H^2(\ua,\QQ)=H^0(\ab,R^2\pi_*\QQ)(-1)\oplus H^2(\ab,\QQ)$ we see that $\Theta$ has zero projection onto the second summand, and thus it generates the first summand, which implies that it is a generator of the stable cohomology of $\ua$ over the stable cohomology of $\ab$, as claimed. Moreover, since the class $\Theta^g$ on $\ua$ is algebraically equivalent to $g!$ times the zero section of the universal abelian variety, see \cite{denmur},\cite{hainnormal},\cite{voisinnotes}, it follows that $\Theta$ is stably algebraically independent with the classes pulled back from $\ab$.
\end{proof}

\section{Gysin exact sequence and the stable cohomology of the partial toroidal compactification}\label{sec:Agpartial}
In this section we set up the method, using the Gysin exact sequence, to compute the cohomology of the union of some partial toroidal compactification and one more stratum, and
we demonstrate how this method works by computing the stable cohomology of Mumford's partial toroidal compactification $\ab'$ of $\ab$.

Recall that Mumford's \cite{mumforddimag} partial toroidal compactification is
the union $\ab'=\ab\sqcup (\ua[g-1]/\i)$ (where $\ua[g-1]/\i$ is still considered as a stack, i.e.~is the universal Kummer family).
The Gysin, also sometimes called  excision long exact sequence for a closed subvariety of a quasi-projective variety
is then the following:
\begin{multline}\label{Gysin}
 \ldots\to \cohc[\ell-1]{\ua[g-1]}\to
\cohc[\ell]{\ab}
\to \cohc[\ell]{\ab'}
\to \\
\to \cohc[\ell]{\ua[g-1]}
\to \cohc[\ell+1]{\ab}\to\ldots
\end{multline}
We remark here that this sequence respects mixed Hodge structures (see \cite[Cor. 5.51]{PetersSteenbrink}).
Since both $\ab'$ and $\ua[g-1]$ are smooth Deligne--Mumford stacks,
it follows from the usual Poincar\'e duality that
the cohomology $H_c^\ell$ is dual to $H^{{\rm top}-\ell}$, where ${\rm top}$ denotes the real dimension of the space.
Noticing that the boundary $\ua[g-1]$
has complex codimension $1$ in $\ab'$, from the above we thus get the dual long exact sequence (where to keep track of things we denote $k=\dim_\RR \ab-\ell=\dim_\RR\ua[g-1]+2-\ell$)
\begin{multline}\label{partialGysin}
 \ldots\rightarrow
\coh[k-2]{\ua[g-1]}(-1)\rightarrow
\coh[k]{\ab'}\rightarrow\\
\rightarrow
\coh[k]{\ab}\rightarrow
\coh[k-1]{\ua[g-1]}(-1)\rightarrow\ldots
\end{multline}

In general, for $k$ and $g$ arbitrary, this exact sequence is non-degenerate.
For instance, the connecting homomorphisms are non-trivial already for $g=3,4$, as described  in
\cite{huto,huto2}.
However, in the stable range the situation is very simple, as all the odd cohomology of each term turns out to be zero, and we immediately obtain the stable cohomology of Mumford's partial toroidal compactification.

\begin{prop}\label{Agpartial}
The stable cohomology of Mumford's partial toroidal compactification $\ab'$  is generated by the classes $\lambda_i$ and  by the class $D$ of the boundary. More precisely,  for any $k<g$, the vector space $H^k(\ab')$ is  generated by degree $k$ monomials in the classes $D,\lambda_1,\lambda_3,\ldots$, where $D$ has degree 2 and $\lambda_{2i+1}$ has degree $4i+2$.
\end{prop}
\begin{proof}
We use the Gysin exact sequence \eqref{partialGysin} above to obtain the stable cohomology of $\ab'$.
Recall that in the stable range the cohomology of $\ab$ was computed by Borel, see Corollary \ref{Ag}, and the stable cohomology of the universal Kummer family is given in Proposition \ref{Xg}. In particular, both of them vanish in odd degree, and thus the long exact sequence \eqref{partialGysin} splits into short exact sequences. We thus obtain
$$
 H^k(\ab')=H^k(\ab)\oplus H^{k-2}(\ua[g-1])(-1)
$$
for $k=2j$, while all the odd-dimensional stable cohomology of $\ab'$ is zero.
For $k=2$ we see that the two generators are $\lambda_1$ for the first summand, and the fundamental class of the boundary ---
which we denote  $D$ --- for the second summand. It follows from \cite[Prop. 1.8]{mumforddimag} or \cite[Lemma 1.1]{vdgeerchowa3} that
$D|_D=-2\Theta$ . The result now follows from Proposition \ref{Xg}.
\end{proof}
\begin{rem}
We observe that the stable cohomology of $\ab'$ is equal to that of $\ua$.
The above proof gives a geometric reason for this:
we consider the inclusion $\ua[g-1]\hookrightarrow\ab'$, and pull back cohomology under it. Then the classes $\lambda_i$ on $\ab'$ pull back to $\lambda_i$ on $\ua[g-1]$, while the class $D$ pulls back to $-2\Theta$.
\end{rem}
While the above proposition does not let us deduce anything about the stabilization of $H^k(\Perf)$ for $k\ll g$, on the dual side we have computed the first few cohomology groups with compact support:
\begin{cor}\label{cor:firststab}
For $g>4$ we have
$$
 \coh[g(g+1)]{\Perf}=\QQ\cdot 1,
\
\coh[g(g+1)-2]{\Perf}=\QQ\cdot \lambda_1^\vee\oplus \QQ\cdot D^\vee,
$$
$$
 \coh[g(g+1)-1]{\Perf}=\coh[g(g+1)-3]{\Perf}=0,
$$
where $\lambda_1^\vee$ and $D^\vee$ denote the images under $\cohc{\ab'}\rightarrow\cohc{\Perf}=\coh{\Perf}$ of the cohomology classes in $\cohc[g(g+1)-2]{\ab'}$ that are Poincar\'e dual to $\lambda_1$ and $D$, respectively.
\end{cor}
\begin{proof}
Since $\ab'$ is smooth, applying Poincar\'e duality the above proposition yields the statement that for $k<g$ the group $H^{g(g+1)-k}_c(\ab')$ is generated by classes dual to $\lambda_{2j+1}$ and to $D$. The complement $\Perf\setminus\ab'$ has real codimension 4 in $\Perf$. Indeed, this is a special case of  Proposition  \ref{prop:basic} (we note that this is special to the perfect cone compactification, and for
example does not hold for the second Voronoi compactification, for which we thus have no result). Hence for $k<4$ we have $H^{g(g+1)-k}_c(\Perf\setminus\ab')=0$ and thus, by
the Gysin exact sequence for $(\Perf\setminus\ab')\subset\Perf$, for $k<4$ we have
$H^{g(g+1)-k}_c(\Perf)=H^{g(g+1)-k}_c(\ab')$. Since $\Perf$ is compact, we finally have
$H^{g(g+1)-k}_c(\Perf)=H^{g(g+1)-k}(\Perf)$ and this gives the corollary.
\end{proof}

\section{Stabilization of the cohomology of $\ua^{\times n}$}\label{sec:Xgn}

In this section we describe the stable cohomology of the $n$'th fiber product $\ua^{\times n}$ of the universal family $\ua\rightarrow \ab$, for a fixed $n$.
It turns out (we thank Ben Moonen for pointing this out and explaining it to us) that the description of the subring in the cohomology of a very general ppav generated by divisors follows from the results of Thompson \cite{thompson} on invariant theory for the symplectic group (this construction is also a special case of a much more general deep construction of Looijenga and Lunts \cite{lolu} in cohomology and of Moonen \cite{moonen} in the Chow ring). The results of Thompson \cite{thompson} are formulated in terms of representations of the symplectic group, which we think of as local systems on $\calA_g$; we give the reformulation in terms of cohomology classes.

Indeed, $\ua^{\times n}$ admits projection maps $p_i:\ua^{\times n}\to\ua$ for $i=1,\dots,n$, and $p_{jk}:\ua^{\times n}\to\ua^{\times 2}$ for $1\leq j<k\leq n$. Let $\Theta\subset\ua$ be the class of the universal theta divisor trivialized along the zero section, and let $P\subset\ua^{\times 2}$ be the class of the universal Poincar\'e divisor trivialized along the zero section. Denote then $T_i:=p_i^*\Theta$, and $P_{jk}:=p_{jk}^* P$. For a very general ppav $A$, the restrictions of these classes to $A^{ n}$ generate the N\'eron--Severi group. We will now prove that these classes freely generate the stable cohomology of $\ua^{\times n}$.
\begin{thm}\label{thm:Xgn}
The cohomology $H^k(\ua^{\times n})$ is independent of $g$ for $k<g$, and as an algebra over the stable cohomology of $\ab$, is generated by the classes $T_i, P_{jk}$. In particular, all stable cohomology classes on $\ua^{\times n}$ are algebraic.
\end{thm}

\begin{proof}
We want to compute the stable cohomology of the $n$-th fiber product of the universal family using the Leray spectral sequence associated with the natural map $\pi^{\times n}:\;\ua^{\times n}\rightarrow\ab $.
Since $\pi^{\times n}$ is a projective map, the Leray spectral sequence degenerates at $E_2$, so that we have
$$
E_\infty^{p,q}=E_2^{p,q}=\cohloc p{\ab}{R^q\pi^{\times n}_*\QQ}
.$$
Recall that the constant local system $\QQ = \VV_0$ is the only one with non-zero stable cohomology (from now on by abuse of language we will call $\VV_0$ the {\em trivial} local system, as it corresponds to the trivial representation). A first consequence of this is that in the stable range $p< g$ the $E_\infty$-terms carry Tate Hodge-structures of weight $p+q$, and hence
\begin{equation}\label{directsum}
\coh[m]{\ua^{\times n}}=\bigoplus_{p+q=m}E_\infty^{p,q}=\bigoplus_{p+q=m}\cohloc p{\ab }{R^q\pi^{\times n}_*\QQ}
\end{equation}
holds for $m\leq g$.

To compute the stable cohomology of $\ua^{\times n}$, it only remains to compute the $\Sp(2g)$-invariant part of $\coh [q]{A^{ n}}$ for an abelian $g$-fold $A$ and $q\leq g$, because this is what contributes the trivial  summands  (which recall, means equal to $\VV_0$) in the local system $R^q\pi^{\times n}_*\QQ$. As local systems of odd weight have zero cohomology, we only have to deal with the case $q=2l$.

The cohomology ring of an abelian variety is the exterior algebra of its first cohomology group, so that we have
$$\coh[2l]{A^{ n}} = \bigwedge^{2l}(\coh[1]A\otimes \QQ^n).
$$
and by \cite[Theorem 3.7]{thompson}, for $2l\leq g$ the $\Sp(2g)$-invariant part of the above cohomology group is  isomorphic to $\Sym^l(\Sym^2\QQ^n)$.

The restriction to $A^n$ of the classes $T_i$, $P_{jk}$ lies in $H^{1,1}(A^{ n})$, and thus the action of the symplectic group on their span is given by the symmetric square of the standard representation of $\Sp(2g)$
acting on $H^1(A^{ n})$.
Therefore, for $l=1$ the $\Sp(2g)$-invariant part of the cohomology is generated by these classes, so that we can identify $\Sym^2\QQ^n$ with the span of the classes $T_i$, $P_{jk}$ and $\Sym^l(\Sym^2\QQ^n)$ with the space of degree $l$ polynomials in $T_i$, $P_{jk}$.
In view of the isomorphism~\eqref{directsum}, this implies that in the stable range $m\leq g$ the classes $T_i, P_{jk}\in\coh[2]{\ua^{\times n}}$ are algebraically independent generators of the cohomology of $\ua^{\times n}$ as an algebra over the stable cohomology of $\ab$.
\end{proof}

\begin{rem}
Thompson's results allow us to describe completely the subalgebra of the rational cohomology of $\ua^{\times n}$ generated by $T_i$, $P_{jk}$, also outside the stable range.
Specifically, Theorem 3.4 of \cite{thompson} corresponds to the statement that the ideal of relations among the classes $T_i$ and $P_{jk}$ is generated by $(g+1)$'st powers of divisors, i.e.~by relations of the form $\left(\sum m_i^2 T_i+\sum m_jm_k P_{jk}\right)^{g+1}=0$,
for arbitrary $m_1,\ldots,m_n\in\ZZ$. In fact \cite[Theorem 3.7]{thompson} describes the cohomology $H^\pu(A^{ n})$ as a representation  of the symplectic group.
\end{rem}
In the following sections, we will often consider the action of the group  $\GL(n,\ZZ)$ on $\ua^{\times n}$.
Indeed, on each fiber $A^{ n}$ of the map $\ua^{\times n}\to\ab$ a matrix $N\in\GL(n,\ZZ)$ acts by the corresponding automorphism $M={}^tN^{-1}$, adding the points together: $(a_1,\ldots,a_n)\mapsto M(a_1,\ldots,a_n)$ considered as points of $A^{ n}$.
Since the classes $T_i$ and $P_{ij}$ lie in $H^{1,1}(A^{ n})$, the action on them is given by the symmetric square of the action of $M$ on $H^1(A^{ n})$. To write it down explicitly, it is convenient to denote $P_{ii}:=2T_i$ for $i=1,\ldots,n$ (and also $P_{ij}:=P_{ji}$ for $i>j$), and then the action is given by the symmetric square of the standard representation of the symmetric group (we are grateful to Ben Moonen and Dmitry Zakharov for discussions on these topics)
\begin{equation}\label{shiftformula}
 M(P_{ij})=\sum\limits_{1\le a\le n}\sum\limits_{1\le b\le n} M_{ia}M_{jb}P_{ab}.
\end{equation}

\section{The perfect cone compactification and details of our approach}\label{sec:basic}
Our results are specific to the perfect cone compactification. As we shall make use of some of the properties of the
perfect cone (also known as first Voronoi) fan decomposition, we shall review this here.
For the original definition of this fan we refer the reader to
\cite{voronoi1}, \cite{voronoi21}, \cite{voronoi22},
for modern treatments and further results
see \cite{namikawabook}, \cite{shepherdbarron}, or \cite{mevi}.

To define the perfect cone decomposition, we consider the open cone $\Sym_{>0}^2(\RR^g)$ of all
real positive definite $g \times g$ matrices and its rational closure $\Sym_{\operatorname{rc}}^2(\RR^g)$,
i.e.~the cone of all semi-positive definite matrices whose kernel is defined over $\QQ$.
The function
$$
\mu: \Sym_{\operatorname{rc}}^2(\RR^g) \to \RR_{>0}
$$
$$
\mu(Q):= \operatorname{min} \{ Q(\xi)  \mid \xi \in \ZZ^g \setminus 0 \}
$$
defines for every $Q \in \Sym_{>0}^2(\RR^g)$ a finite and non-empty set
$$
M(Q):= \{ \xi \in \ZZ^g \mid Q(\xi)=\mu(Q) \}.
$$
The perfect cone decomposition is then given by the union of the convex hulls
$$
\sigma(Q):= \sum_{\xi \in M(Q)}\RR_{\geq 0}{}^t\xi\xi.
$$
The group $\GL(g,\ZZ)$ operates on the collection of these cones with a finite number of orbits.
In the proof that the cohomology stabilizes we use a special property of the perfect cone decomposition,
namely the fact that the number of
codimension $i$ strata stabilizes, i.e.~is independent of $g$ if $i \leq g$.
This follows easily from the above definition
(see Proposition \ref{prop:basic}).

This is indeed a special property of the perfect cone decomposition. In particular, one cannot expect
that the second Voronoi decomposition has stable cohomology. Recall that the boundary divisors of a toroidal
compactification
correspond to the set of orbits under $\GL(g,\ZZ)$ of $1$-dimensional cones in the corresponding fan. There is only one such
cone in the perfect cone decomposition, namely the square of a primitive linear form. Let $l(g)$ be the number of
inequivalent $1$-dimensional cones in the second Voronoi decomposition.
Note that $l(g)$ is the number of components of the boundary of $\Vor$, so that $l(g)=\dim H^{g(g+1)-2}(\partial\Vor)$ holds. As $H^{g(g+1)-1}_c(\ab)$ vanishes
(it is Poincar\'e dual to $H^1(\ab)$), the Gysin long exact sequence associated with the inclusion of the boundary into $\Vor$ implies that $H^{g(g+1)-2}(\Vor)$ surjects onto $H^{g(g+1)-2}(\partial\Vor)$, so that $\dim H^{g(g+1)-2}(\Vor)\geq l(g)$ holds.
It is well known that $l(2)=l(3)=1$, then we have $l(4)=2$ \cite{voronoi1}, 
\cite{erry2}, $l(5)=9$ \cite{degr}, while $l(6) \geq 20,000$ \cite{degr}.

In general, and this was pointed out to us by V.~Alexeev, we have at least  $l(g) \geq g-3$. This estimate comes from the root lattices $D_n$. Indeed,
the quadratic form associated to such a root lattice defines a second Voronoi cone and by the results of Baranovskii  and Grishukhin \cite{BG} the
barycentric rays of these cones for $4 \leq n \leq g$ give independent rigid forms.

We have already introduced the stratification of $\Perf$ into the closed subvarieties $\beta_i$ which
lie over $\ab[g-i]^{\operatorname{Sat}}$  under the map
$\varphi: \Perf \to \Sat$, see \eqref{equ:beta}. Recall that $\beta_i^0=\beta_i  \setminus \beta_{i-1}$.
As for any toroidal compactification, the locally closed sets $\beta_i^0$  are further stratified
into strata $\stratum{} \sigma \subset \beta_i$, corresponding to the orbits of rank $i$ cones $\sigma$. More precisely let
$\sigma \subset \Sym^2_{\operatorname{rc}}(\RR^i)$
be a rank $i$ cone in the perfect cone decomposition. Given such a cone $\sigma$, one associates with it
a torus bundle $q(\sigma): \calT(\sigma) \to \ua[g-i]^{\times i}$. The fiber of the torus bundle $q(\sigma)$ is the torus
$\Torus_i/\Torus_{\sigma}$ where $\Torus_i= \Sym^2(\ZZ^i) \otimes \CC^*$ and $\Torus_{\sigma} \subset \Torus_i$ is given by
$\Torus_{\sigma} = (\Span( \sigma) \cap \Sym^2(\ZZ^i)) \otimes \CC^*$.
Denoting by $p_i:\ua[g-i]^{\times i}\to\ab[g-i]$ the
universal $i$-fold product, we thus have a double fibration
$\pi(\sigma)=p_i \circ q(\sigma) : \calT(\sigma) \to \ua[g-i] \times
\ldots \times \ua[g-i] \to \ab[g-i]$.
The stratum associated to
$\sigma$ is then equal to the quotient $\stratum{}\sigma=G(\sigma) \backslash \calT(\sigma)$ where
$G(\sigma)$ is the stabilizer of $\sigma$ in $\GL(i,\ZZ)$.
We have
$\beta_i^0=\bigsqcup_{{\rm all\ }\sigma{\rm \,of}{\rm \, rank} \, i}\beta(\sigma)$ and
$\beta_i=\bigsqcup_{{\rm all\ }\sigma{\rm \,of}{\rm \, rank} \, \geq i}\beta(\sigma)$.
Recall that the complex codimension of $\stratum{}{\sigma}$ in $\Perf$ is equal to $\dim \sigma$.

In the previous sections we discussed the topology of the partial compactification $\ab'$. We now want to add
further boundary strata and at this point it becomes important to us that we make use of specific properties of the
perfect cone decomposition.
The most important property of the perfect cone decomposition for our purposes is that the stratum $\beta_i$ has complex codimension
$i$ within $\Perf$, a fact we have already used in Corollary \ref{cor:firststab}. As we have pointed out before, the situation is very different
for example for $\Vor$, where boundary divisors appear arbitrarily deep into the boundary.

\begin{prop}\label{prop:basic}
The following holds for the perfect cone decomposition $\Perf$:\\
(i) The codimension of $\beta_i$ is equal to $i$.\\
(ii) Let $\ell$ be an integer.
If $g \geq \ell$, then the number of strata $\stratum{}\sigma$ of codimension $\ell$ in
$\Perf$ is given by an integer $N(\ell)$  independent of $g$.
\end{prop}
\begin{proof}
(i) We first show that the codimension of $\beta_i$ is at most $i$. This follows since the standard cone
$\langle x_1^2, \ldots , x_i^2 \rangle$ belongs to the perfect cone decomposition, has dimension $i$ and rank $i$.
Conversely, consider a cone $\sigma \subset \Sym^2_{\operatorname{rc}}(\RR^i)$ in the perfect cone decomposition
of rank $i$. Since the rays of $\sigma$ are spanned by rank $1$ matrices, there must be at least
$i$ independent generators of $\sigma$ and thus the dimension of $\sigma$ is at least $i$. Therefore, the same
holds for the codimension of $\stratum{}\sigma$.\\
(ii) Let $\sigma$ be a cone which gives rise to a stratum of codimension $\ell$, i.e.~assume that $\sigma$  is of dimension $\ell$, and rank $i$.
Then $i \leq \ell$.  Choose $i$ rays in $\sigma$ such that the corresponding linear forms
are independent (over $\QQ$). These linear forms generate a (not necessarily saturated) sublattice $L$ in $\ZZ^g$; let
$L'$ be its saturation. Since the general element in $\sigma$ has rank $i$ it follows that $\sigma \subset \Sym^2_{\operatorname{rc}}(L' \otimes \RR)$.
After acting by a suitable element in $\GL(g,\ZZ)$ we can assume that $L'$ is the sublattice of $\ZZ^g$ spanned by the first $i$
unit vectors and thus that $\stratum{}\sigma\subset \beta_i^0$ and in particular  $\stratum{}{\sigma} \subset \Perf \setminus \beta_{\ell+1}$.
Hence these strata are enumerated by the $\GL(m,\ZZ)$-orbits of the cones in the perfect cone decompositions
of $\Sym^2(\ZZ^m)$ for all integers $m\leq \ell$. Clearly, the number $N(\ell)$ of such orbits is independent of $g$, for $g\ge\ell$.
\end{proof}

\smallskip
As we already said, our approach is that we use Gysin sequences to successively compute the cohomology of $\Perf$. We start with the
strata $\beta(\sigma)$ associated to rank $i$ cones $\sigma$ to compute the cohomology of $\beta_i^0$ and then keep going deeper into the boundary to prove results
about the cohomology of $\Perf$ itself.

As we have just seen,  $\stratum{}{\sigma}=G(\sigma) \backslash \calT(\sigma)$ with
$\pi(\sigma)=p_i \circ q(\sigma) : \calT(\sigma) \to \ua[g-i]^{\times i} \to \ab[g-i]$.
We will need to compute the (stable) cohomology of strata  $\stratum{}{\sigma}$ in several cases.
For this we must recall
the construction of the torus bundle  $\calT(\sigma)$ in more detail.  We first of all fix the cusp $U$ over which we work. We
shall want to work with the standard cusps, i.e.~we fix $U$ as the isotropic subspace of $\QQ^{2g}$ spanned by the first $i$
elements of the standard basis. The parabolic subgroup $P(U)$ of $\Sp(2g,\ZZ)$ which fixes $U$  is generated by elements of the
following form: The first set of generators is
$$
g_1=
\begin{pmatrix}
{\bf 1}_i& 0 & S & 0\\
0 & {\bf 1}_{g-i} & 0 & 0\\
0 & 0 & {\bf 1}_i & 0\\
0 & 0 & 0 & {\bf 1}_{g-i}
\end{pmatrix}, \text{ where }
S={}^tS \in \Sym^2(\ZZ^i).
$$
These matrices generate the center $P'(U)$ of the unipotent radical of $P(U)$ and dividing out by this (normal) subgroup
gives $\Torus_i \times \CC^{i(g-i)} \times \HH_{g-i}$ where $\Torus_i = \Sym^2(\ZZ^i) \otimes_\ZZ \CC^*$.

The second set of generators consists of elements of the form
$$
g_2=
\begin{pmatrix}
{\bf 1}_i& 0 & 0 & 0\\
0 & A & 0 & B\\
0 & 0 & {\bf 1}_i & 0\\
0 & C & 0 & D
\end{pmatrix}, \text{ where }
\begin{pmatrix}
A & B\\
C & D
\end{pmatrix} \in \Sp(2(g-i),\ZZ),
$$
and
\begin{equation}\label{g3}
g_3=
\begin{pmatrix}
{\bf 1}_i& M & 0 & N\\
0 & {\bf 1}_{g-i} & {}^tN & 0\\
0 & 0 & {\bf 1}_i & 0\\
0 & 0 & -{}^tM & {\bf 1}_{g-i}
\end{pmatrix}, \text{ where }
M,N \in \Mat_\ZZ(i,g-i).
\end{equation}
Note that the elements of type $g_2,g_3$ generate a Jacobi group
the quotient by which is $\calT_i \to \ua[g-i]^{\times i}
\to \ab[g-i]$, where the $i$-fold universal family
$\ua[g-i]^{\times i}\to \ab[g-i]$ is the quotient of
$\CC^{i(g-i)} \times \HH_{g-i}$ by the Jacobi group,
and the fiber of the first projection is isomorphic to $\Torus_i$.
The cone $\sigma$ defines a subtorus $\Torus_{\sigma}$ of $\Torus_i$ and correspondingly a subbundle
$\calT_{\sigma}$ of $\calT_i$, with  $\calT(\sigma)=\calT_i/\calT_{\sigma}$.

For later use, it is also useful to consider the torus bundle $\calT_i^\vee\rightarrow\ua[g-i]^{\times i}$, whose fiber is the torus $\Torus_i^\vee=\Sym^2(\ZZ^i)^\vee\otimes\CC^*$ dual to $\Torus_i$.

\begin{prop}\label{p:identification}
For $1\leq j\leq k\leq i$, let us define line bundles $\calS_{jk}$ on $\ua[g-i]^{\times i}$ by setting $\calS_{jk}=p_{jk}^*(P^{-1})\otimes L^{-1/2}$ if $j<k$ and $\calS_{jk}=p_j^*(\Theta^{-2})\otimes L^{-1}$ if $j=k$, where $L$ denotes the pull-back of the line bundle of modular forms on $\ab[g-i]$.
Then the torus bundle $\calT_i^\vee\rightarrow\ua[g-i]^{\times i}$ is isomorphic to the fiber product of the $\CC^*$-bundles $\calS_{jk}^0$ obtained by removing the $0$-section from the $\calS_{jk}$.
\end{prop}

\begin{proof}
Let us denote by $\{\tau_{r,s}\}_{1\leq r\leq s\leq i}$ the basis of $\Sym^2(\RR^i)^\vee$ dual to the basis $\{x_rx_s+x_sx_r\}_{1\leq r\leq s\leq i}$ and by $M_i$ the lattice generated by the $\tau_{r,s}$. (We will later use the notation $M$ for this lattice, when there is no longer a danger that this might be confused  with the matrices $M$, which appear later 
in this proof and whose notation has also become standard in the literature). 
Then  $s_{jk}:=e^{2\pi\;\sqrt{-1}\tau_{jk}}$ define coordinates on
 the fiber $M_i\otimes\CC^*\cong(\CC^*)^{\frac{i(i+1)}2}$ of $\calT_i^\vee$. Next, we look at the transformation law for the $s_{jk}$ under the subgroup $(\ZZ^{2(g-i)i})\cong \Mat_\ZZ(i,g-i)^{\oplus 2}$ of the Jacobi group generated by transformations \eqref{g3} of type $g_3$. For all $1\leq j\leq k\leq i$ and $M,N \in \Mat_\ZZ(i,g-i)$, one gets
\begin{multline}\label{transform}
(M,N)s_{jk}\\
=e^{2\pi\sqrt{-1}\bigl(\sum\limits_{l=1}^{g-i}(m_{jl}\tau_{k,g-i+l}+m_{kl}\tau_{j,g-i+l})+\sum\limits_{\alpha,\beta=1}^{g-i}m_{j\alpha}m_{k\beta}\tau_{g-i+\alpha,g-i+\beta}\bigr)}s_{jk}.
\end{multline}
In particular, the matrix $N$ acts trivially on $s_{jk}$.

To prove the claim, it suffices to show that $s_{jk}$ is a local section of $\calS_{jk}$ for all $1\leq j\leq k\leq i$.
If $j=k$ holds, this transformation behavior agrees with that of $\theta(\Omega,Z)^{-2}$ for $\Omega=(\tau_{g-i+\alpha,g-i+\beta})\in \HH_{g-i}$ and $Z=(\tau_{k,g-i+1},\dots,\tau_{k,g})\in\ZZ^i$, hence $s_{jj}$ is a local section of the bundle $p_j^*(\Theta\otimes L^{1/2})$.

For $j<k$ one sees from \eqref{transform} that $\calS_{jk}$ is the pull-back of the line bundle $\calS_{12}'$, which is defined
analogously to $\calS_{j,k},$ in the special case $j=1, k=2$ and $i=2$,
over $\ua[g-i]\times_{\ab[g-i]}\ua[g-i]$.
Again, comparing with the transformation behavior of the theta function yields that the restriction of $\calS_{12}'$ to the fiber $A\times A$ of $\ua[g-i]^{\times 2}$ is the inverse of the Poincar\'e bundle twisted by $L^{1/2}$. Indeed, for every $[A]\in\ab[g-i]$,  the restriction of $\calS_{12}'$ to $\{0\}\times A$ and $A\times \{0\}$ is trivial, whereas the restriction of $\calS_{12}'$ to the diagonal $\ua[g-i]\hookrightarrow \ua[g-i]^{\times 2}$ gives $(\Theta\otimes L^{1/2})^{-2}$.
 \end{proof}

Finally the parabolic subgroup $P(U)$ contains elements of the form
$$
g_4=
\begin{pmatrix}
{}^tQ^{-1}& 0 & 0 & 0\\
0 & {\bf 1}_{g-i} & 0 & 0\\
0 & 0 & Q & 0\\
0 & 0 & 0 & {\bf 1}_{g-i}
\end{pmatrix}, \text{ where }
Q \in \GL(i,\ZZ).
$$
In order to obtain $\stratum{}{\sigma}$ from $\calT(\sigma)$ we consider all $Q$ such that
the action on the space $\Sym^2(\ZZ^i)$ given by
$$
\GL(i,\ZZ) \ni Q:\ X\mapsto {}^tQ^{-1}XQ^{-1}.
$$
maps the cone $\sigma$ to itself. This gives us a finite group $G(\sigma)$ and  $\stratum{}{\sigma}= G(\sigma)  \backslash \calT(\sigma)$.
At this point we would also like to point out that the Jacobi group is a normal subgroup of $P(U)/P'(U)$ and that elements
of the form $g_2$ and $g_4$ commute. We will use these facts without mentioning them explicitly when first dealing with the cohomology of $\calT(\sigma)$ and only then taking the part invariant under $G(\sigma)$.

To compute the cohomology of one such stratum $\stratum{}{\sigma}$, one uses two Leray spectral sequences, namely those for
the torus bundle $q(\sigma)$ and for $\pi(\sigma): \calT(\sigma)\rightarrow\ab[g-i]$.
Indeed, in a first step, for $[A] \in \ab[g-i]$ the cohomology of $\pi(\sigma)^{-1}([A])$ can be computed from the Leray spectral sequence of the torus bundle restricted to the fiber $p_i ^{-1}([A])\cong A^i$ of the universal family $\ua[g-i]^{\times i}$:
\begin{equation}\label{lerayfiber}
E_2^{p,q}(q(\sigma))=H^q(\Torus_i/\Torus_{\sigma})\otimes H^p(A^{i})
\Longrightarrow H^{p+q}\big(\pi(\sigma)^{-1}([A])\big).
\end{equation}
Since the group $G(\sigma)$ acts fiberwise, one can then compute the $G(\sigma)$-invariant part
$\tilde E_2^{p,q}(q(\sigma))$
of the cohomology
of the fibers of $\pi(\sigma)$.
Varying the fiber over $\ab[g-i]$, one thus
obtains a direct sum of local systems on $\ab[g-i]$.
In the second step of the argument we then use the Leray spectral sequence for the map $\pi(\sigma)$.
To write $E_2^{p,q}(\pi(\sigma))=\cohloc q {\ab[g-i]}{R^p\pi(\sigma)_*\QQ}$, we consider $E_2^{p,q}(q(\sigma))$ as  giving rise to a spectral sequence of local systems converging to $R^p\pi(\sigma)_*\QQ$.
Provided one can control the differentials of the spectral sequence and one knows the cohomology
of the local systems, from the $G(\sigma)$-invariant terms $\tilde E_2^{p,q}(q(\sigma))$ one can thus compute the cohomology of the stratum $\stratum{}{\sigma}$.

Our aim is to compute stable cohomology. This simplifies the situation considerably since stable cohomology only comes from
trivial local systems.  In other words we only have to take into consideration the part $\hat E_2^{p,q}(q(\sigma))$  of $\tilde E_2^{p,q}(q(\sigma))$ which
gives rise to trivial local systems $\VV_0$ on  $\ab[g-i]$. This allows us to work with a smaller spectral sequence, which still converges to the stable cohomology of $\stratum{}{\sigma}$. Moreover, we will be able to argue that in the stable range not only the cohomology of $\ua^{\times i}$ but also that of every open torus bundle is zero in any odd degree. This will drastically simplify dealing with the spectral sequences.

\section{Stable cohomology of strata}\label{sec:stabletorusbundles}
The aim of this section is to prove a stabilization result for the cohomology of the strata $\stratum{}\sigma$ of toroidal compactifications of $\ab$.
This kind of stabilization  occurs independently of the choice of the compactification.
To set up notation, let us assume that we have an admissible collection $\ssigma=\{\Sigma_g\}_{g\in\NN}$ of admissible fans $\Sigma_g$ in $\Sym^2_{\operatorname{rc}}(\RR^g)$ or in a $\GL(g,\ZZ)$-invariant open subset of $\Sym^2_{\operatorname{rc}}(\RR^g)$.
This means that for all $g<g'$,
the intersection of $\Sigma_{g'}$ with $\Sigma_{g}$  coincides with $\Sigma_g$, where we identify $\RR^g$ with the subspace of $\RR^{g'}$ generated by
the first $g$ coordinate vectors.
Then for each $g$, we define  $\Asigma$ to be the (possibly partial) toroidal compactification of $\ab$ defined by this admissible collection of fans.

Let $\sigma$ be a rank $i$ cone of dimension $\ell$ in $\ssigma$. Then $i\leq \ell$ and $\sigma$ defines a stratum $\stratum{g}{\sigma}\subset\Asigma$ for any genus $g \geq i$.
This  is the quotient of a torus bundle $\calT_g(\sigma)$ over
$\ua[g-i]^{\times i}$ by a finite group
$G(\sigma)$.  The rank of the torus fiber is $i(i+1)/2-\ell$.
The group $G(\sigma)$ and the fiber of the torus bundle do not depend on  $g$, but  $\stratum{g}{\sigma}$ itself does.

In what follows we must be very careful as to which space we are working in, and will thus keep the subscript $g$ everywhere.
\begin{thm} \label{thm:stablestrata}
For a given cone $\sigma$ the cohomology groups $H^k(\calT_g(\sigma))$ and $H^k(\stratum g\sigma)$ stabilize for $k<g-i-1$.

Moreover, the cohomology in this range is algebraic, and explicitly one has
$$
H_\stab^\pu(\calT_g(\sigma))\cong\QQ[\lambda_1,\lambda_3,\dots]\otimes \Sym^\pu (\Span(\sigma)\cap\Sym^2(\QQ^i)),
$$
where the generators of $\Span(\sigma)$ have degree $2$.
The stable cohomology of $\stratum{}\sigma$ is the invariant part of the stable cohomology of $\calT_g(\sigma)$ under the natural action of the stabilizer $G(\sigma)$ of the cone $\sigma$ in $\GL(i,\ZZ)$.
\end{thm}

\begin{rem}\label{rem:classes}
More precisely, there is an isomorphism
$$H_\stab^\pu(\calT_g(\sigma))\cong\QQ[\tau_{rs},\lambda_{2m+1}|1\leq r\leq s\leq i, m\in\ZZ]/(\sigma^\perp\cap M)$$
where $\{\tau_{rs}\}$ denotes the basis of $M=\Sym^2(\ZZ^i)^\vee$ dual to $\{x_rx_s+x_sx_r\}$ and $\sigma^\perp=\{\phi\in\Sym^2(\RR^i)^\vee| \phi(\xi)=0\  \forall \xi\in\sigma\}$ is the orthogonal complement of $\sigma$.

One should interpret this isomorphism as a description of
the stable cohomology of $\calT_g(\sigma)$ as a quotient of the stable cohomology of $\ua[g-i]^{\times i}$, by identifying $\tau_{rs}$ with the class $-2T_r\in H_\stab^2(\ua[g-i]^{\times i})$ if $r=s$ and with the class $-P_{rs}\in H_\stab^2(\ua[g-i]^{\times i})$ if $r\neq s$.
\end{rem}
\begin{proof}
Let us recall that for $g<g'$ the stabilization map on the moduli spaces of abelian varieties is induced by the map $\ab\rightarrow \ab[g']$ given by mapping $[A]\in\ab$ to $[A\times B]\in\ab[g']$, where $B$ is a fixed abelian variety of dimension $g-g'$. (Different choices of $B$ lead to the same map in cohomology). For the universal family, the stabilization map is given by mapping $(A,x)\in\ua$ to $(A\times B,x\times 0_B)$, where $0_B\in B$ denotes the identity element of $B$.
These stability maps can be lifted to $\stratum g\sigma$ to obtain the following commutative diagram
 \begin{equation}\label{diagram-beta}
\xymatrix@C=66pt
{
\calT_g(\sigma)\ar[r]\ar[d]^{q(\sigma)}\ar@/_3ex/[dd]_{\pi(\sigma)}&{\calT_{g'}}(\sigma)\ar[d]_{q(\sigma)}\ar@/^3ex/[dd]^{\pi(\sigma)}\\
\ua[g-i]^{\times i}\ar[r]
\ar[d]^{p_i}
&{\ua[g'-i]^{\times i}}\ar[d]_{p_i}\\
{\ab[g-i]}\ar[r]&{\ab[g'-i]},
}
\end{equation}
where the map $\calT_g(\sigma)\rightarrow \calT_{g'}(\sigma)$ is well-defined due to the fact that the fibers of the maps $q(\sigma)$ are independent of the genus $g$. The finite automorphism group $G(\sigma)$, being the stabilizer of $\sigma$ in its $\QQ$-span, which is isomorphic to
$\Sym^2_{\geq 0}(\QQ^i)$, does not depend on $g$ and acts equivariantly on the diagram, thus defining the stabilization map $\stratum g\sigma\rightarrow \stratum {g'}{\sigma}$.

Let us consider the Leray spectral sequence associated with $\pi(\sigma)$, with $E_2$ term
$
E^{p,q}_2=\cohloc p{\ab}{R^q\pi(\sigma)_*\QQ}.
$
By the stability Theorem \ref{thm:stable}, in the stable range $p<g-i$ over $\ab[g-i]$ the cohomology of the non-trivial symplectic local systems $\VV_{\ud\mu}$ vanishes, hence $E_2^{p,q}$
only depends on the trivial summands of the local system  $R^q\pi(\sigma)_*\QQ$, or, equivalently,
by the part of the cohomology of the fiber of $\pi(\sigma)$ that is invariant under the action of the symplectic group.
In Lemma~\ref{l:cohpsia} below we will show that this cohomology
stabilizes in degree $p< g-i$ and that in this range it is isomorphic to the truncation of the symmetric algebra of a $\QQ$-vector space $M_\QQ/W_\QQ$ which is isomorphic to the $\QQ$-span of the extremal rays of the cone $\sigma$.

In particular, the stabilization map induces an isomorphism between the $E_2$ terms in the range $p,q<g-i$ of the Leray spectral sequences associated with $\pi(\sigma)$ for $g$ and for $g'$.
Next, we observe that the generators of $M_\QQ/W_\QQ$ represent algebraic classes in the cohomology, hence the $E^{p,q}_2$ with $p,q<g-i$ vanish if $p+q$ is odd and carry Tate Hodge-structures of weight $(p+q)/2$ if $p+q$ is even.
The vanishing for odd $p+q$  implies that all differentials $d_r:\;E^{p,q}_r\rightarrow E^{p+r,q-r+1}_r$ with $p+r<g-i$ and $q<g-i$ are zero, so that $E_2=E_\infty$ holds. In particular, this is the case for $p+q<g-i-1$, and therefore for $k<g-i-1$ we have
$$
\coh[k]{\calT_{g'}(\sigma)}=\coh[k]{\calT_g(\sigma)}=\bigoplus_{p+q=k}\cohloc p{\ab[g-i]}{R^q\pi(\sigma)_*\QQ},
$$
which in view of Lemma \ref{l:cohpsia} is $0$ for odd $k$ and isomorphic to
$$
\bigoplus_{p'+q'=k/2}\coh[2p']{\ab[g-i]}\otimes\Sym^{q'}(M_\QQ/W_\QQ)$$
for even $k$.

This proves the claim for the cohomology of $\calT_{g}(\sigma)$.

The strata $\stratum g\sigma$ and $\stratum{g'}\sigma$ are the quotient of $\calT_g(\sigma)$, respectively, of $\calT_{g'}(\sigma)$ by the finite group $G(\sigma)$. Therefore, the  cohomology of $\stratum g\sigma$ (resp., $\stratum {g'}\sigma$) is the $G(\sigma)$ invariant part of the cohomology of $\calT_g(\sigma)$ (resp., of $\calT_{g'}(\sigma)$).
Since the diagram \eqref{diagram-beta} is $G(\sigma)$-equivariant,  the cohomology of $\stratum g\sigma$ stabilizes in the same range as that for $\calT_g(\sigma)$. The explicit description of the stable cohomology and its algebraicity come from taking the $G(\sigma)$-invariant part in the description of the stable cohomology of $\calT_g(\sigma)$.
The proof is thus completed by proving the following lemma.
\end{proof}
\begin{lm}\label{l:cohpsia}
In degree less than $g-i$, the $\Sp(2g-2i)$-invariant part of the cohomology of  the fiber $\Psi_{g,A}=\pi(\sigma)^{-1}([A])\subset\calT_g(\sigma)$ is algebraic and independent of $g$. In particular, the cohomology vanishes in odd degree.

Furthermore, if we denote  the basis of $\Sym^2(\RR^i)^\vee$ dual to the basis $\{x_rx_s+x_sx_r\}_{1\leq r\leq s\leq i}$ by $\{\tau_{r,s}\}_{1\leq r\leq s\leq i}$,  the lattice generated by the $\tau_{r,s}$ by $M$ and the intersection $\sigma^\perp\cap M$ by $W$, in even degree $k<g-i$ the cohomology of $\Psi_A$ is isomorphic to the degree $k/2$ part of the  symmetric algebra generated by the quotient $M_\QQ/W_\QQ$ with $M_\QQ=M\otimes_\ZZ \QQ$, $W_\QQ = W\otimes_\ZZ \QQ$.
\end{lm}
\begin{proof}
To proceed, we want to describe more precisely the torus bundle $\calT_g(\sigma)$.
We recall that its fiber is the torus $\Torus_i/\Torus_{\sigma}$,
where $\Torus_i= \Sym^2(\ZZ^i) \otimes \CC^*$ and $\Torus_{\sigma} \subset \Torus_i$ is given by $\Torus_{\sigma} = (\Span (\sigma) \cap \Sym^2(\ZZ^i)) \otimes \CC^*$. Thus duality defines a canonical isomorphism between the quotient $\Torus_i/\Torus_{\sigma}$
and the torus
$\Torus_{\sigma^\perp}=(\sigma^\perp\cap M)\otimes \CC^*$.
This enables us to view $\calT_g(\sigma)$ as a subbundle of the dual torus bundle $\calT_i^\vee$ of Proposition~\ref{p:identification}.

Let us choose a $\ZZ$-basis $\phi_1,\dots,\phi_m$  ($m=i(i+1)/2$) of $W:=\sigma^\perp\cap M$ and write  $\phi_j=\sum_{1\leq k_1\leq k_2\leq i} \alpha_{j,k_1,k_2}\tau_{k_1,k_2}$.
Then setting $w_j:=e^{2\pi\sqrt{-1}\phi_j}=\prod_{1\leq k_1\leq k_2\leq i} s_{k_1,k_2}^{\alpha_{j,k_1,k_2}}$ for $j=1,\dots,m$ defines a set of parameters for the fiber
 $\Torus_{\sigma^\perp}$ of the torus bundle $\calT_g(\sigma)$.
Note that by Proposition \ref{p:identification} each $w_j$ defines a local section of the bundle $L_j:=\otimes \calS_{k_1k_2}^{\alpha_{j,k_1,k_2}}$ over $\ua[g-i]^{\times i}$.
Hence, if we denote  the complement of the $0$-section by $L_j^0$, the torus bundle $\calT_g(\sigma)$ is contained in the direct sum of the $L_j$ as
$$\calT_g(\sigma) = L_1^0\times_{\ua[g-i]^{\times i}} \dots\times_{\ua[g-i]^{\times i}} L_m^0.$$

Now, let us consider  $\Psi_{g,A}$, which by definition  is the restriction to $A^{\times i}\cong p_i^{-1}([A])$ of the torus bundle $\calT_g(\sigma)$.
The stabilization map commutes with $\pi(\sigma)$, hence its restriction to $A^i$ induces a map $\Psi_{g,A}\rightarrow \Psi_{g',A'}$, where $A'=A\times B$ is given by the image of the point $[A]\in\ab[g-i]$ under the stabilization map $\ab[g-i]\rightarrow\ab[g'-i]$.
To study the cohomology of $\Psi_{g,A}$
and $\Psi_{g',A}$,
we use the  Leray spectral sequence of the torus bundle  $\Psi_{g,A}\rightarrow A^{i}$  and $\Psi_{g',A'}\rightarrow {A'}^{i}$, which we denote by $E^{p,q}_\pu$ and $E'_\pu{}^{p,q}$, respectively.
The $E_2$-terms are of the form
$$E_2^{p,q}=\coh[q]{\Torus_{\sigma^\perp}}\otimes\coh[p]{A^{i}}
=
\bigwedge^q W_{\QQ}\otimes\bigwedge^p{\coh[1]A^{ i}},$$
where we used the isomorphism $W_\QQ \cong\coh[1]{\Torus_{\sigma^\perp}}$.
Let us denote by $E_{r,\inv}^{p,q}$ the $\Sp(2g-2i)$-invariant part of $E_r^{p,q}$, and by $E'_{r,\inv}{}^{p,q}$ the $\Sp(2g'-2i)$-invariant part of $E'_r{}^{p,q}$. Then the description of $E_2$ given above, combined with Theorem \ref{thm:Xgn}, implies that the stabilization map $\Psi_{g,A}\rightarrow\Psi_{g',A'}$ induces an isomorphism $E_{2,\inv}^{p,q}\cong E_{2,\inv}^{p,q}$ if $p\leq g-i$.
Furthermore, for $p\leq g-i$ the term $E_{2,\inv}^{p,q}$ vanishes for $p$ odd and is given by
$$
E_{2,\inv}^{p,q}=\bigwedge^qW_\QQ \otimes \QQ[T_k,P_{k_1k_2}]_{p/2}
\cong \bigwedge^qW_\QQ \otimes \Sym^{p/2}M_\QQ
$$
for $p$ even. Here we identified the symmetric algebra $\Sym^rM_\QQ$ with $\QQ[T_k,P_{k_1k_2}]_{r}$ using the map $\tau_{kk}\mapsto -2T_k$, $\tau_{kk'}\mapsto -P_{kk'}$ for $k<k'$.

Furthermore, in view of the structure of $\calT_g(\sigma)$ as product of $\CC^*$-bundles, the Leray spectral sequence of $\calT_g(\sigma)$ degenerates at $E_3$, and the $d_2$-differentials are determined by the Euler classes of the $\CC^*$-bundles, i.e.~by the Chern class of the line bundles $L_1,\dots,L_m$.
By construction, and by the description of the bundles $\calS_{jk}$ given in Proposition \ref{p:identification}, one has $c_1(L_j)=-\sum_{1\leq k_1<k_2\leq i}\alpha_{j,k_1,k_2}P_{k_1k_2} -2\sum_{1\leq k\leq i}T_k$, where the coefficients of this linear combination are independent of $g$. In particular, in the stable range also the differentials in the spectral sequences $E_{\pu,\inv}$ and $E'_{\pu,\inv}$ coincide.

We can rephrase the description of $d_2$ in terms of multilinear algebra by saying that
$$
E_{2,\inv}^{2r,q}=\bigwedge^qW_\QQ \otimes \Sym^{r}M_\QQ
\xrightarrow{d_2}
E_{2,\inv}^{2r+2,q-1}=\bigwedge^{q-1}W_\QQ \otimes \Sym^{r+1}M_\QQ
$$
is the differential of the degree $r+q$ strand of the dual Koszul complex associated to the inclusion $W_\QQ\hookrightarrow M_\QQ$ of $\QQ$-vector spaces (see \cite[A2.6.1]{eisenbudbook}), provided both $E_{2,\inv}^{2r,q}$ and $E_{2,\inv}^{2r+2,q-1}$ are in the stable range. As this Koszul complex is exact, this immediately yields $E_{3,\inv}^{p,q}=0$ for $q\geq1$ in the range $p+q< g-i$, as well as
$$
E_{3,\inv}^{k,0}= \Sym^{k/2}M_\QQ/(\phi_1,\dots,\phi_m) = \Sym^{k/2}(M_\QQ/W_\QQ)
$$
for even $k<g-i$ and the vanishing of $E_{3,\inv}^{k,0}$ for odd $k<g-i$.
\end{proof}
\subsection{Standard cones}\label{subsec:standard}
We close this section by illustrating Theorem \ref{thm:stablestrata} in the concrete case of the $i$-dimensional standard cone
$$\sigma=\langle x_1^2,x_2^2,\dots,x_i^2\rangle.
$$
As the rank of $\sigma$ is equal to $i$, we get a commutative diagram
$$
\xymatrix@C=66pt
{
\calT(\sigma)\ar[r]^{(\CC^*)^m\text{-bundle}}
\ar[rd]^{\pi(\sigma)}
\ar[d]_{/G(\sigma)}
&{\ua[g-i]^{\times i}}\ar[d]\\
{\stratum g{\sigma}}\ar[r]_{\pi(\sigma)^{G(\sigma)}}&{\ab[g-i]}
}
$$
where the rank of the torus bundle is $m=\binom i2$.
To proceed, we describe more precisely the torus bundle $\calT(\sigma)$ and the group $G(\sigma)$. As explained in the previous section, the fiber of $\calT(\sigma)$ is given by the torus $W\otimes_\ZZ\CC^*$, where $W$ denotes the integral points of the orthogonal complement of $\sigma$.
As  $\sigma^\perp$ is spanned by $\tau_{jk}$ for all $1\leq j<k\leq i$, one has
$$
\sigma^\perp = \Span(-\tau_{12},\dots,-\tau_{i-1,i}).
$$
From Proposition \ref{p:identification} it follows that exponentiating the coordinate $-\tau_{jk}$ gives rise to a local section of the Poincar\'e bundle $P_{jk}$. Therefore, we have that $\calT(\sigma)$ is the fiber product of the Poincar\'e bundles $P_{jk}$ over $\ua[g-i]^{\times i}$ with the $0$-section removed.

The stabilizer $G(\sigma)$ of the standard cone in $\GL(i,\ZZ)$ is generated by sign changes and permutations of the coordinates $x_1,\dots,x_i$. In particular, its action on $\Span(\sigma)\cap \Sym^2(\QQ^i)$ factors through the action of the symmetric group $S_i$ permuting $x_1^2,\dots,x_i^2$. If we identify $\Span(\sigma)$ with the quotient of the dual space $\Sym^2(\RR^i)^\vee=\Span(\tau_{11},\tau_{12},\dots,\tau_{ii})$ by $\sigma^\perp=\Span(\tau_{jk}|\;j<k)$, we get the standard representation of $S_i$ on $\Span(\tau_{11},\dots,\tau_{ii})$.

From this it follows that we can identify the stable cohomology of the torus bundle $\calT(\sigma)$ with the quotient of the stable cohomology of $\ua[g-i]^{\times i}$ by the Euler classes $P_{12},\dots,P_{i-1,i}$ of the factors of $\calT(\sigma)$, or, equivalently, with the subalgebra of the stable cohomology of $\ua[g-i]^{\times i}$ generated by $T_1,\dots,T_i$. As the stable cohomology of $\stratum{}\sigma$ is the $S_i$-invariant part of the stable cohomology of $\calT(\sigma)$, we get the following result:
\begin{lm}\label{lm:standard}
For the standard cone $\sigma = \langle x_1^2,\dots,x_i^2\rangle$ the cohomology groups  $H^k(\stratum g\sigma)$ stabilize for $k<g-i-1$. The stable cohomology of $\stratum g\sigma$, as an algebra over the stable cohomology of $\ab[g-i]$, is freely generated
by classes $\epsilon_j\in H^{2j}(\stratum g\sigma)$ for $1\leq j\leq i$, where $\epsilon_j$ can be viewed as the degree $j$ symmetric polynomials in the classes $T_1,\dots,T_i\in H^2(\ua[g-i]^{\times i})$.
\end{lm}

\section{Stabilization of cohomology of $\Perf$}\label{sec:stable}
We can  turn our attention to the (open) strata $\beta^0_i$ in the perfect cone compactification $\Perf$, which are disjoint unions of the strata described in the previous section. To compute their cohomology, we will use the Gysin exact sequence.
We note that each individual stratum $\stratum{}{\sigma}$ is a finite quotient of a smooth variety, namely the quotient of $\calT_g(\sigma)$ by $G(\sigma)$,
and hence Poincar\'e duality holds between cohomology and cohomology with compact support of complementary degree.
We would like to point out that this is no longer true for the strata $\beta^0_i$ themselves as they will, in general, be singular, and thus we will now want to work with cohomology with compact support, in close to top degree.

In the proof of Theorem \ref{thm:stablestrata}, we observed that the cohomology of $\stratum{}\sigma$ stabilizes with respect to a well-defined map $\stratum g\sigma\rightarrow\stratum{g'}\sigma$.
Clearly, these maps extend to a morphism $\Perf \rightarrow \Perf[g']$. This follows from the fact that in the toroidal construction, also the gluing of the strata $\stratum{}\sigma$ commutes with the stabilization morphisms induced by $\ab\rightarrow\ab[g']$. Hence, by restriction we also get maps $\Perf\supset\betazero{i,g}\rightarrow \betazero{i,g'}\subset \Perf[g']$ which induce the pullback maps $H^k(\betazero{i,g'})\rightarrow H^k(\betazero{i,g})$. However, if $\betazero{i,g}$ is singular, there is no natural associated map $H_c^{\topd-k}(\betazero{i,g'})\rightarrow H_c^{\topd-k}(\betazero{i,g})$, due to the fact that Poincar\'e duality may not hold for $\betazero{i,g}$.

\begin{rem}\label{rem:extensions}
As explained in the introduction, this is the first section the results of which do not apply to an arbitrary toroidal compactification. To simplify notation and statements, we will formulate everything for the perfect cone toroidal compactification, and make use of Proposition \ref{prop:basic}, which is specific to the perfect cone compactification. However, we would like to point out that in fact the results below hold in greater generality: indeed, property (ii) of Proposition \ref{prop:basic} follows from combining admissibility with property (i). Hence every admissible collection $\ssigma$ satisfying the property (i) from the statement of Proposition \ref{prop:basic} defines a sequence of toroidal compactifications $\{\Asigma\}$ to which our stability results (Proposition \ref{prop:stablebeta} and Lemma \ref{lemma:setminusbeta}) extend. Thus our main result, the stabilization and algebraicity given by Theorems \ref{thm:main} and \ref{algebraic} apply for any such $\ssigma$,  possibly after replacing cohomology with cohomology with compact support in the case that $\Asigma$ is a partial compactification.

As we will see at the end of this section, some natural examples of such sequences of partial compactifications are the matroidal partial compactification $\Matr$ (and applying the machinery below gives the stabilization results in this case), as well as the smooth locus $\Perf[g,\operatorname{smooth}]$ or the simplicial locus $\Perf[g,\operatorname{simp}]$ within $\Perf$.
\end{rem}

\begin{rem}\label{rem:homology}
Throughout this section, we prove all our results in the case of cohomology. However, it is straightforward to adapt the proofs to work also for homology. For this we need only to
replace cohomology with compact support with Borel--Moore homology, and the Gysin long exact sequences with their duals, the long exact sequences in Borel--Moore homology associated with closed inclusions.
\end{rem}

\begin{prop} \label{prop:stablebeta}
The strata  $\beta^0_{i,g}$  have stable cohomology with compact support in degree close to the top degree as $g$ goes to infinity.

More precisely, the cohomology
groups $H_c^{\pu}(\beta^0_{i,g}, \QQ)$ satisfy
\begin{equation}\label{plus-strata}
\cohc[k]{\beta^0_{i,g}}=\bigoplus_{\rank\sigma = i}\cohc[k]{\stratum g{\sigma}}
\end{equation}
if $k>\topd-g+i+1$.
Furthermore, in this range the cohomology groups with compact support are independent of $g$ and are all algebraic, so that in particular all odd cohomology vanishes.
\end{prop}
\begin{proof}
We first recall from Theorem \ref{thm:stablestrata} that the cohomology of the strata $\stratum g\sigma$ stabilizes in degree $k<g-i-1$ and is algebraic in this range. As each $\stratum g\sigma$ is the global quotient of the smooth space $\calT_g(\sigma)$ by a finite group, Poincar\'e duality holds and the cohomology with compact support of $\stratum g\sigma$ stabilizes and is algebraic in degree $k>2\dim_\CC\stratum g\sigma-g+i+1\leq\topd-g+i+1$. Hence,  equality \eqref{plus-strata} implies the rest of the claim.

By definition of the toroidal compactification, the locus $\betazero{i,g}\subset\Perf$ is stratified by the locally closed subvarieties $\stratum g\sigma$ defined by the rank $i$ cones $\sigma$ in the perfect cone decomposition. The dimension of these cones ranges from $i$ to $i(i+1)/2$. Therefore, we have
$$
\betazero{i,g}=\bigsqcup_{0\leq j\leq i(i+1)/2-i}W_{j,g}
$$
where we denoted by $W_{j,g}$ the (disjoint) union of all $\stratum g\sigma$ with $\sigma$ of rank $i$ and dimension $i(i+1)/2-j$. Note that the closures $\overline W_{j,g}$ define a filtration on $\betazero{i,g}$ and that the Gysin spectral sequence associated with this filtration has $E_1$ term
$$
E_1^{p,q}=\cohc[p+q]{W_{p,g}}
$$
where we set $W_{p,g}$ to be empty if no cone of rank $i$ and dimension $i(i+1)/2-p$ exists. As the cohomology with compact support of $W_{p,g}$ is the direct sum of the cohomology with compact support of its locally closed strata $\stratum g\sigma$, to show the claim \eqref{plus-strata} it suffices to show that the spectral sequence associated with $\{\overline W_{j,g}\}$ degenerates at $E_1$ in the range $p+q>\topd-g+i+1$.

To this end, let us note that Theorem \ref{thm:stablestrata} implies that $E_1^{p,q}$ vanishes if $p+q$ is odd and $p+q>2\dim_\CC W_{p,g}-g+i+1$, i.e.~for $q>g^2-i^2-2i+p+1$. In particular, in the non-trivial columns, where $i(i+1)/2-i\leq p \leq i(i+1)/2$ holds, all differentials of the form
$$
E_r^{p,q}\rightarrow E_r^{p+r,q-r+1} \ \text{or} \ E_r^{p-r,q+r-1}\rightarrow E_r^{p,q}
$$
with $p+q>2\dim_\CC\betazero{i,g}-g+i+1 = g^2-i+1$ are in this range. Hence, either the source or the target space of the differential vanishes. From this it follows that $E_1^{p,q}=E_{\infty}^{p,q}$ holds for $p+q>\topd -g+i+1$.
\end{proof}

\begin{rem}
The proposition above holds for any admissible collection $\ssigma$. The proof can be easily extended to this more general case. This only requires us to keep track of the fact that the top degree may be larger than $g(g+1)-2i$.
\end{rem}
\begin{lm}\label{lemma:setminusbeta}
The cohomology of the open subset $\Perf\setminus\beta_{i+1,g}$ stabilizes in close to the top degree, i.e.~the cohomology group $\cohc[g(g+1)-k]{\Perf\setminus\beta_{i+1,g}}$ is independent of $g$ for $k<g$.
Furthermore, in this range the cohomology groups with compact support are all algebraic, so that in particular all odd cohomology vanishes.
\end{lm}

\begin{proof}
The main idea of the proof is the same as in the previous proposition: we consider an appropriate stratification of $\beta_{i+1,g}$ into locally closed subsets, we prove that the $E_1^{p,q}$-terms of the associated Gysin spectral sequence stabilize for $p+q<2\dim_\CC\Perf - g = g^2$ and that moreover  all differentials with either source or target in this stable range vanish, so that they stabilize as well. In this case, the natural approach is to stratify $\Perf\setminus\beta_{i+1,g}$ as the union of the strata $\betazero{i-j,g}$ for $j=0,\dots,i$. Then the associated Gysin spectral sequence in cohomology with compact support has $E_1$ term
$$
E_1^{p,q}=\cohc[p+q]{\betazero{i-p,g}},
$$
which in view of Proposition \ref{prop:stablebeta} stabilizes and is algebraic for $p+q>2\dim_\CC\betazero{i-p,g}-g+i-p+1=g^2-i+p+1$. In the case of the strata $\beta_0=\ab$ and $\beta_1=\ua$, however, the bound given in Theorem \ref{thm:stable} and Proposition \ref{Xg} is slightly better, so that we have
$p+q>2\dim_\CC\betazero{i-p,g}-g+i-p=g^2-i+p$ as stability range for $p\in\{i-1,i\}$. At this point, we observe that all $E_1^{p,q}$-terms with either $p+q>g^2$ or $p+q>g^2-1$ and $p<i$ lie in the stable range, and that they vanish if $p+q$ is odd. This implies that $E_1^{p,q}=E_\infty^{p,q}$ stabilizes for $p+q>g^2$, so that we have $\cohc[k]{\Perf\setminus\beta_{i+1,g}}=\bigoplus_{0\leq j\leq i}\cohc[k]{\betazero{j,g}}$ for $k>g^2$.
This is enough to prove the stability of cohomology with compact support for $k>g^2$. The fact that the classes are algebraic follows from the corresponding results for the $\betazero{j,g}$.
\end{proof}

We can now finally prove that the cohomology of $\Perf$ with compact support, in degree close to top, stabilizes.
The method is similar to the one developed in the previous sections: to compute $H^{{\rm top}-k}_c(\Perf,\QQ)$,
we need to analyze all the strata of complex codimension up to $\lfloor k/2\rfloor$ in $\Perf$. As we have pointed out before, the
mere fact that for $g\gg k$ there is a finite fixed collection of such cones,
which was shown in  Proposition \ref{prop:basic},
is special to the perfect cone decomposition. We are now ready to prove our main result, the stabilization of
cohomology $H^{g(g+1)-k}(\Perf, \QQ)$ for $k < g$.

\begin{proof}[Proof of the main theorem, Theorem \ref{thm:main}]
As  each $\beta_i$ has codimension $i$ in $\Perf$ and $\Perf$ is compact, the cohomology  of $\Perf[g']$ in degree $g'(g'+1)-k$ coincides with the cohomology with compact support of $\Perf[g']\setminus\beta_{\lceil g/2\rceil+1,g'}$ for $k < g\leq g'$. Then the claim follows from the isomorphism  $$\cohc[g(g+1)-k]{\Perf\setminus\beta_{\lceil g/2\rceil+1,g}}\cong \cohc[g'(g'+1)-k]{\Perf\setminus\beta_{\lceil g/2\rceil+1,g'}}$$ described in Lemma~\ref{lemma:setminusbeta} for $k<g<g'$.
\end{proof}

The singularities of the space $\Perf$ and the ensuing failure of Poincar\'e duality  have forced us to switch to cohomology with compact support.
We will now discuss open subsets of $\Perf$ where this problem does not arise.
The first question is to understand the singularities of $\Perf$ better. To be precise, we are interested in those singularities which come from the perfect cone compactification itself, rather than those that arise from the non-neatness of the group $\Sp(2g,\ZZ)$, which do not give singular points of the stack $\Perf$ but only of its coarse moduli space.
As the latter singularities do not occur on suitable level covers they are no
obstruction to Poincar\'e duality as long as one works with rational coefficients. We shall denote the locus of these
singularities that are not resolved by going to a level cover by $\Psing$. Indeed this is the singular locus of the stack $\Perf$.
We have the following recent result of Dutour Sikiri\'c, Sch\"urmann, and the second author:
\begin{prop}[see \cite{DutourHulekSchuermann}]\label{prop:boundsingular}
The stack $\Perf$ is smooth for $g \leq 3$ and the (complex) codimension of its singular locus $\Psing$
is equal to $10$ for any $g \geq 4$.
\end{prop}
We denote the underlying variety of the smooth locus of the stack $\Perf$ by
$$
 \Perf[g,{\operatorname{smooth}}]:=\Perf \setminus  \Psing.
$$
From the toroidal point of view, we can view $\Perf[g,{\operatorname{smooth}}]$ as the partial toroidal compactification of $\ab$ defined by the cone decomposition one obtains by considering only those perfect cones that are basic. We recall that a cone is called basic if its generators form a $\ZZ$-basis of $\Sym^2(\ZZ^g)$. In fact recall more generally that a cone is called simplicial if its generators form a $\RR$-basis of $\Sym^2(\RR^g)$, and in this case the toric variety
is locally the quotient of a smooth space by a finite abelian group.

The union of all simplicial cones defines an open subset $\Perf[g,{\operatorname{simp}}]$ of $\Perf$. Since the singularities of $\Perf$ in codimension $10$
arise from the non-simplicial cone $D_4$ (see  \cite[Theorem 1 (ii) ]{DutourHulekSchuermann}), it follows
that the codimension of the complement of $\Perf[g,{\operatorname{simp}}]$ in $\Perf$ is also $10$.
The main advantage of working with the simplicial locus (and suitable open subsets) is that all its points are rationally smooth.
This follows from rational smoothness of simplicial toric varieties, see eg.~\cite[Theorem 11.4.8]{coxlittleschenck}.
Note that rational smoothness ensures that rational cohomology coincides with the middle perversity intersection cohomology.
In our case, this implies that the cohomology of the simplicial locus satisfies Poincar\'e duality and that we have a cycle map to cohomology which is a ring homomorphism, i.e.~we can interpret algebraic cycles of (complex) codimension $k$ as cohomology classes in degree $2k$.

\begin{prop} \label{prop:cohossmooth}
\begin{enumerate}
\item \label{cohossmooth-i}
The cohomology stabilizes for the smooth and for the simplicial locus within $\Perf$, i.e.~the cohomology groups $H^k(\Perf[g,\operatorname{smooth}])$ and $H^k(\Perf[g,\operatorname{simp}])$ are both independent of $g$ for $k<g$.
\item 
For $k < 19$ there are isomorphisms
$$
H^{\operatorname{top}-k}(\Perf, \QQ)
\cong H^k(\Perf[g,{\operatorname{smooth}}],\QQ) \cong H^k(\Perf[g,{\operatorname{simp}}],\QQ)
$$
induced by the Poincar\'e duality on $\Perf[g,{\operatorname{smooth}}]$ and $\Perf[g,{\operatorname{simp}}]$, respectively.
\end{enumerate}
\end{prop}
\begin{proof}
(i)
As explained in Remark \ref{rem:extensions}, the proof of the main result above also serves to show that $\cohc[g(g+1)-k]{\Perf[g,\operatorname{simp}]\setminus\beta_{i+1,g}}$ is independent of $g$ for $k<g$. In particular, this holds for $i=\lceil g/2\rceil+1$. As the codimension of $\beta_i$ is $i$, one gets
$$
\cohc[g(g+1)-k]{\Perf[g,\operatorname{simp}]\setminus\beta_{i+1,g}}
\cong
\cohc[g(g+1)-k]{\Perf[g,\operatorname{simp}]}
\cong
\cohc[k]{\Perf[g,\operatorname{simp}]}
$$
where the last isomorphism is Poincar\'e duality for the rationally smooth $\Perf[g,\operatorname{simp}]$. This shows the stabilization of the cohomology of $\Perf[g,\operatorname{simp}]$. The proof for $\Perf[g,\operatorname{smooth}]$ is completely analogous.

(ii)
We first note that $H^{\operatorname{top}-k}(\Perf, \QQ) \cong H_c^{\operatorname{top}-k}(\Perf, \QQ)$ holds since $\Perf$ is compact.
By the Gysin sequence applied to the inclusion of
$\Psing$ into $\Perf$ we obtain an isomorphism $H_c^{\operatorname{top}-k}(\Perf, \QQ)
\cong H_c^{\operatorname{top}-k}(\Psmooth,\QQ)$ for
$k < 2\cdot\codim_\CC \Psing - 1=19$.
Finally, we have  by Poincar\'e duality $H_c^{\operatorname{top}-k}(\Psmooth,\QQ)\cong H^k(\Psmooth,\QQ)$. The same proof applies
to the simplicial locus since Poincar\'e duality also holds there.
\end{proof}

The third open locus of $\Perf$ which is of interest to us is the {\em matroidal} locus $\Matr$. The importance of this locus was pointed out by
Melo and Viviani \cite{mevi}, who identified it as the biggest partial compactification of $\ab$  contained in both the second Voronoi and the
perfect cone compactification. This means that we can think of $\Matr$ as the ``intersection'' of $\Vor$ and $\Perf$. The matroidal locus
is defined as the partial compactification obtained by taking all matroidal cones. Recall that a matrix $A \in \Mat_{\ZZ}(g,n)$ is called totally unimodular if
every square submatrix has determinant $-1$, $0$ or $1$. A matrix $A \in \Mat_{\ZZ}(g,n)$ is called {\em unimodular} if there exists a matrix $B \in \GL(g,\ZZ)$
such that $BA$ is totally unimodular. A cone in $\Sym^2_{\operatorname {rc}}(\RR^g)$ is called {\em matroidal} if it is spanned by the rank $1$ forms
defined by the columns of a unimodular matrix. It is known that all matroidal cones are simplicial \cite[Theorem 4.1]{errydicing} and thus $\Matr \subset \Perf[g,{\operatorname{simp}}]$.
We also know that the codimension of the complement of $\Matr$ in $\Perf$ is $5$, due to the existence of a non-matroidal dimension $5$ cone in genus $5$ (namely the
cone $\sigma_{NS}$ discussed below among the codimension 5 strata).

As matroidal cones are simplicial, we have that $\Matr$ is again rationally smooth. In particularly it satisfies Poincar\'e duality (with rational coefficients), so that the same argument as in the  proof of part (i) of Proposition \ref{prop:cohossmooth} applies to $\Matr$, thus providing a proof of the stabilization of the rational cohomology of $\Matr$ in degree $k<g$ (Theorem \ref{thm:Matrstabilizes}).

Finally, the considerations above also apply to even smaller open loci of $\Perf[g,{\operatorname{simp}}]$. For instance, one can take the partial compactification of $\ab$ given by taking the union of all strata associated with standard cones, i.e.~of all $\stratum g{\sigma}$ with $\sigma$ a cone of the form $\sigma=\langle x_1^2,x_2^2,\dots,x_i^2\rangle$ for $0\leq i\leq g$ (see \S\ref{subsec:standard}). We will denote this union of the standard strata by $\Std$. As there is just one standard cone in each dimension, and standard cones are always basic, it is easy to adapt the proof of part (i) of Proposition~\ref{prop:cohossmooth} to prove that the rational cohomology of $\Std$ stabilizes in degree $k<g$ and is generated by algebraic classes. However, as the stable cohomology of strata associated with standard cones is known by Lemma~\ref{lm:standard}, in this case we can compute this stable cohomology explicitly.
\begin{thm}\label{thm:Stdstabilizes}

The cohomology  of the partial toroidal compactification defined by the standard cones stabilizes, i.e.~$H^k(\Std,\QQ)$ does not depend on $g$ for $k<g$. The stable cohomology is the polynomial algebra generated by the odd $\lambda$-classes and the fundamental classes $[\beta_i]\in\coh[2i]{\Std}$ of the boundary strata.
\end{thm}

\begin{proof}
As remarked in the introduction, the stable cohomology of $\Std$ coincides with the cohomology of the inductive limit $\Std[\infty]$ of the sequence $\Std[g] \rightarrow \Std[g+1]$ defined by taking products with a fixed element of $\ab[1]$.  Let us observe that there is a well-defined product $\Std[g_1]\times\Std[g_2]\rightarrow\Std[g_1+g_2]$ for all $g_1,g_2\geq 0$. These products define a structure of H-space on $\Std[\infty]$, so in particular its cohomology is a commutative and associative graded Hopf algebra over $\QQ$. Hence, by Hopf's theorem, the stable cohomology of $\Std$ is a free graded-commutative algebra.
However, in the case of $\Std$ we know that the stable cohomology is concentrated in even degree, so that the stable cohomology is a polynomial algebra. At this point, it only remains to identify the generators.

Let us recall from Lemma~\ref{lm:standard} that the stable cohomology of the stratum $\betazero i\cap\Std$ is isomorphic to the polynomial algebra $\QQ[\eta_j,\lambda_{2k+1}|\;1\leq j\leq i, k\geq 0]$ generated by the odd $\lambda$-classes and by $i$ other classes $\eta_j\in\coh[2j]{\betazero i\cap\Std}$.
By the Gysin exact sequence associated with the stratification $\{\betazero i\cap\Std\}$ of $\Std$ we have
\begin{equation}\label{eq:stdsplit}
\coh[k]{\Std} \cong \bigoplus_{i\geq 0}\coh[k-2i]{\betazero i\cap\Std}(-i)
\end{equation}
in the stable range $k<g$.

Combining this with Lemma~\ref{lm:standard} we obtain that the rank of the stable cohomology in degree $k$ concides with the rank of the polynomial algebra $\QQ[\eta_j,\lambda_{2k+1}|\;j,k\geq 0]$ with $\deg\eta_j=2j$, $\deg\lambda_{2k+1}=4k+2$. Therefore, to prove the claim it suffices to notice that for all $i\geq 1$, the fundamental class of $[\beta_i]$ is not a product of classes $[\beta_j]$ with $j<i$ and $\lambda$-classes. This is indeed the case, as $[\beta_j]\in \coh[2j]{\Std}$ vanishes under the pull-back of the open inclusion $\Std\setminus \beta_i\hookrightarrow\Std$. Note that $[\beta_i]\neq 0$ follows from \eqref{eq:stdsplit} and the degeneration at $E_1$ of the Gysin exact sequence associated with the stratification of $\Std$ by boundary strata.
\end{proof}

\section{Automorphisms and the stable cohomology of the next stratum}\label{sec:next}
To further demonstrate that our method can give explicit results, in this section we will compute the stable cohomology of the  ``second partial'' compactification of $\ab$
obtained by adding to $\ab'$ the locus of semiabelic varieties of torus rank 2 (which now come in two flavors, depending on whether the toric part is $\PP^1\times\PP^1$ or two copies of $\PP^2$, so that we have two strata to deal with).
Note that this part is still the same for the perfect cone, matroidal, second Voronoi, and central cone toroidal compactifications.

More precisely, the perfect cone decomposition of $\Sym^2_{\geq 0}\RR^g$ contains exactly two $\GL(g,\ZZ)$ orbits of cones whose general element is a form of rank $2$, namely the orbits of the cones
$$\sigma_{1+1} := \langle x_1^2,x_2^2\rangle\quad\text{ and }\quad
\sigma_{K_3} := \langle x_1^2,x_2^2,(x_1-x_2)^2\rangle.
$$
This implies that the locus within $\Perf$ of semiabelic varieties of torus rank $2$ is the union of an open stratum $\stratum{}{\sigma_{1+1}}$, where the normalization of the corresponding semiabelic variety is an irreducible $\PP^1\times\PP^1$ bundle, and a closed stratum $\Delta:=\stratum{}{\sigma_{K_3}}$. In the following, we will determine the stable cohomology of these strata and of their union using Theorem \ref{thm:stablestrata} and the Gysin exact sequence.
Both these strata are fibrations over $\ua[g-2]\times_{\ab[g-2]}\ua[g-2]$.
By Theorem \ref{thm:Xgn} the stable cohomology $H^k(\ua[g-2]^{\times2})$ for $k<g-2$ is generated by the classes $T_1,T_2$ of the two pullbacks of the theta divisor, and the class $P:=P_{12}$ of the universal Poincar\'e divisor, all trivialized along the zero section.

For the open stratum $\stratum{}{\sigma_{1+1}}$ we know from  \cite[p. 356]{mumforddimag}, see also
\cite[Section 5]{huto2},
that it is the quotient by automorphisms of the total space of the universal Poincar\'e line bundle $\calP\to \ua[g-2]^{\times2}$, with its zero section removed (where the Poincar\'e bundle is trivialized along the zero section $\ab[g-2]\to \ua[g-2]^{\times2}$).
This description indeed agrees with that given in Section \ref{subsec:standard} for the $i$-dimensional standard cone, in the case $i=2$. Lemma \ref{lm:standard} gives us the following result:
\begin{lm}\label{lm:1+1}
The cohomology  of ${\stratum{}{\sigma_{1+1}}}$ stabilizes in degree $k < g-3$.
More precisely, in this range
$\coh[k]{\stratum{}{\sigma_{1+1}}}$, as an algebra over the stable cohomology of $\ab[g-2]$, is isomorphic to the polynomial algebra $\QQ[T_1+T_2,T_1T_2]$ on two generators, of degrees $2$ and $4$, respectively.
\end{lm}

Our approach to the locally closed stratum $\Delta = \stratum{}{\sigma_{K_3}}$ is analogous. In this case, the toroidal description yields that $\Delta$ is the quotient of $\calX^{\times 2}_{g-2}$ by the group $G(\sigma_{K_3})$ generated by the following three involutions:
\begin{align}
\label{invK3-1}
(x_1,x_2)&\leftrightarrow (-x_1,-x_2)\\
\label{invK3-2}
(x_1,x_2)&\leftrightarrow (x_2,x_1)\\
\label{invK3-3}
(x_1,x_2)&\leftrightarrow (x_1,x_1-x_2).
\end{align}

Note that the involution \eqref{invK3-1} acts trivially on  $\Sym^2(\RR^2)$, whereas  \eqref{invK3-2} can be viewed as the involution  $x_1^2\leftrightarrow x_2^2$ and \eqref{invK3-3} as $x_2^2\leftrightarrow (x_1-x_2)^2$. From this it follows that the action of $G(\sigma_{K_3})$ on $\Span(\sigma_{K_3})$ factors through the standard representation of the symmetric group $S_3$ on the generators of $\sigma_{K_3}$. Let us recall from Theorem \ref{thm:stablestrata} that the stable cohomology of $\stratum{}{\sigma_{K_3}}$ is the $G(\sigma_{K_3})$-invariant part of the symmetric algebra on the generators of $\sigma_{K_3}$, tensored with $H^\pu_\stab(\ab[g-2])$. If we denote by $(\alpha_1,\alpha_2,\alpha_3)=(x_1^2,x_2^2,(x_1-x_2)^2)$ the $\ZZ$-basis given by the generators of $\sigma_{K_3}$ and by
$$(\gamma_1,\gamma_2,\gamma_3)=(\tau_{11}+\tau_{12},\tau_{12}+\tau_{22},-\tau_{12})
$$
the dual basis, we have
$$
H^\pu_\stab(\stratum{}{\sigma_{K_3}})
\cong
H^\pu_\stab(\ab[g-2])\otimes \left(\Sym^\pu(\QQ\gamma_1+\QQ\gamma_2+\QQ\gamma_3)\right)^{S_3},
$$
so that by the theory of symmetric functions the stable cohomology of $\stratum{}{\sigma_{K_3}}$ is freely generated by the elementary symmetric functions in the $\gamma_j$.
The geometric meaning of these generators can be made more explicit by using the correspondence between the exponentials of the coordinates $\tau_{jk}$ and the classes $T_j,P_{jk}\in H^2_\stab(\ua[g-2]^{\times 2})$ coming from Remark \ref{rem:classes} and Proposition \ref{p:identification}. This yields the following description of the three generators $\xi,\eta,\zeta$:
\begin{align*}
-\gamma_1-\gamma_2-\gamma_3=-\tau_{11}-\tau_{12}-\tau_{22} &\mapsto&\xi&=2(T_1+T_2)+P,\\
\gamma_1\gamma_2+\gamma_2\gamma_3+\gamma_3\gamma_1=\tau_{11}\tau_{22}-\tau_{12}^2
&\mapsto&\eta&=4T_1T_2-P^2,\\
-\gamma_1\gamma_2\gamma_3=(\tau_{11}+\tau_{12})(\tau_{12}+\tau_{22})\tau_{12}
&\mapsto&\zeta&=P(2T_1+P)(2T_2+P).
\end{align*}

This proves the following result:
\begin{lm}\label{lm:K3}
The cohomology of $\Delta=\stratum{}{\sigma_{K_3}}$ stabilizes in degree $< g-3$, and in this range is generated over the stable cohomology of $\ab$ by the classes $\xi$, $\eta$ and $\zeta$ that have degrees $2$, $4$ and $6$, respectively.
\end{lm}

Since each of the two substrata of $\beta_2^0$ are smooth, we can translate our results into cohomology with compact support, and using the Gysin spectral sequence we can thus compute
the cohomology $H_c^{\topd-k}(\beta_2^0,\QQ)$ with compact support in the stable range $k<g-3$ (where $\topd:=g(g+1)-4$ is the (real) dimension of $\beta_2^0$. Recall that the stratum $\betazero 2$ is smooth, as all rank $2$ cones are basic. In particular, Poincar\'e duality gives an isomorphism $\coh[k]{\betazero 2}\cong \cohc[\topd-k]{\betazero 2}$, so that we can state stability results for $\betazero 2$ directly in terms of cohomology.

For later use we notice in particular that
\begin{cor}\label{cor:stablecohbeta2}
For $g>11$, the Betti numbers of $\betazero 2$ in even degree are as follows:
$$
\begin{array}{|r|rrrrr|}
\hline
k&0&2&4&6&8\\
\hline
&&&&&
\\[-2.2ex]
\dim H^{k}(\betazero 2,\QQ)&1&3&6&11&19\\[0.3ex]
\hline
\end{array}
$$
Moreover, the stable cohomology vanishes in odd degree $k\leq 8$.
\end{cor}

\section{Further computations: stable cohomology of $\Perf$ in degree up to 12}
\label{sec:numbers}
In this section we outline the technical difficulties encountered in extending the explicit computations of stable cohomology to higher degree, and list the results of this computation for the next couple of cases. As a result, we compute $H^{\topd-k}(\Perf,\QQ)$
for $k\le 12$, proving Theorem \ref{thm:Perfnumbers} (and then from the computations also easily deduce $H^{\topd-k}_c(\Matr,\QQ)$, proving Theorem \ref{thm:Matrnumbers}).  To do this, for each of the (many) cones, we will list the rank $1$ forms generating  it (as in \cite[Chapter 4]{vallentinthesis})
and the automorphism group preserving the cone (for most cases these have been computed by the second and third authors in \cite{huto} and \cite{huto2}, we provide the couple extra computations necessary). We then describe the action of the automorphism groups on the cohomology of the torus fiber.

We also recall that from Lemma \ref{l:cohpsia} and Proposition \ref{prop:stablebeta} and their proofs we know that the cohomology of each stratum is purely algebraic, all odd cohomology vanishes, that $E^{p,q}$ vanishes for $p$ odd, and that the Leray spectral sequence for the map to $\ab[g-k]$ degenerates at $E_2$. Thus our job amounts to computing the invariant part of the cohomology of each toric fiber, and then following the method of Lemma \ref{l:cohpsia} and Proposition \ref{prop:stablebeta}.

\subsection{Strata of codimension 3}
There is only one stratum of $\Perf$ of (complex) codimension 3 that we have not considered yet; it is the standard
degeneration of torus rank 3, given by the cone
$$\sigma_{1+1+1}=\langle x_1^2,x_2^2,x_3^2\rangle.$$

In this case, we can apply Lemma \ref{lm:standard} for rank $i=3$, which gives us that the stable cohomology of $\stratum{}{\sigma_{1+1+1}}$ is freely generated by the elementary symmetric polynomials in the $T$-classes
$$T_1+T_2+T_3, \ T_1T_2+T_2T_3+T_3T_1, \ T_1T_2T_3$$
and the odd $\lambda$-classes.
Recall that to compute $H^{\ge g(g+1)-12}(\Perf)=H^{\le 12}(\Perf)$ (see
Proposition \ref{prop:cohossmooth}),
we will only need $H^{\le 6}$ of this stratum.
The dimensions of the stable cohomology are thus given by
$$
\begin{array}{l|llll}
k&
0&2&4&6\\\hline
\dim
H_{\operatorname{stable}}^{k}(\stratum{}{\sigma_{1+1+1}},\QQ)&
1&2&4&8
\end{array}$$

\subsection{Strata of codimension 4}
We have three strata of codimension 4, of which one (the standard cone) has torus rank 4, and two others correspond to torus rank 3 degenerations, i.e.~define strata in $\beta_3^0$ (we refer to \cite{grhu2} for the detailed description of all strata of codimension up to 5, and of course to \cite{huto2} for more details). The standard torus rank 4 cone is
$$\sigma_{1+1+1+1}=\left<x_1^2,x_2^2,x_3^2,x_4^2\right>.$$
As above, we can apply Lemma~\ref{lm:standard} to $\stratum{}{\sigma_{1+1+1+1}}$, which yields the following values for the dimension of the stable cohomology in degree $k\leq 4:$
$$
\begin{array}{l|lll}
k&
0&2&4\\\hline
\dim
H_{\operatorname{stable}}^{k}(\stratum{}{\sigma_{1+1+1+1}},\QQ)&
1&2&4
\end{array}$$

\smallskip
The other two cones of codimension 4 have torus rank 3. One is
$$\sigma_{K_3+1}=\left<x_1^2,x_2^2,(x_1-x_2)^2,x_3^2\right>.$$

In this case  $\calT(\sigma_{K_3+1})$ is a torus bundle of rank $2$, with parameters $s_{1,3}^{-1},s_{2,3}^{-1}$. Therefore $\calT(\sigma_{K_3+1})$ is isomorphic to a product of the Poincar\'e bundles $(P_{1,3}\otimes L^{1/2})^0$ and $(P_{2,3}\otimes L^{1/2})^0$ with the $0$-section removed.
The reduced automorphism group of $\sigma_{K_3+1}$, i.e.~the automorphism group divided by  $\pm {\bf 1}$,
was computed in \cite[Lemma 6]{huto}, and is equal to $S_3\times (\ZZ/2\ZZ)$. Its action on $\Span(\sigma)$ factors through the action of $S_3$ permuting the first three generators of $\sigma$ and fixing the last one.
Then Theorem \ref{thm:stablestrata} implies that the stable cohomology of $\stratum{}{\sigma_{K_3+1}}$ is isomorphic to an algebra
$H^\pu_\stab(\ab)\otimes \QQ[f_2,g_2,g_4,g_6]$
where the subscript identifies the degree of the free generators. The generator $f_2$ can be identified with $x_3^2$, whereas $g_{2i}$ corresponds to the degree $i$ elementary polynomial in $x_1^2,x_2^2,(x_1-x_2)^2$.

Using the same approach and notation as in Lemma \ref{lm:K3}, this yields the isomorphism
$$
H^\pu_\stab(\stratum{}{\sigma_{K_3+1}}) = \QQ[T_{3},\xi,\eta,\zeta,\lambda_{2m+1}|\;m\in\ZZ]\subset H^\pu_\stab(\ua^{\times 3})
$$
for the classes $\xi=2(T_1+T_2)+P$, $\eta=4T_1T_2-P^2$, $\zeta = P(2T_1+P)(2T_2+P)$.
This yields the following formula:
$$
\begin{array}{l|ccc}
k&0&2&4\\\hline
\dim H_{\operatorname{stable}}^{k}(\stratum{}{\sigma_{K_3+1}},\QQ)
&1&3&7
\end{array}
$$

\smallskip
Finally, we have the last codimension 4 cone given by
$$\sigma_{C_4}=\left<x_1^2,x_2^2,(x_1-x_3)^2,(x_2-x_3)^2\right>.$$

This cone was studied in \cite[Section 5.4]{huto}: a natural choice of parameters for $\calT(\sigma_{C_4})$ is given by $s_{12}^{-1},s_{13}s_{23}s_{33}$;
the automorphism group of $\sigma_{C_4}$ is $S_4$, and it is generated by the three involutions sending the point $(x_1,x_2,x_3)\in \RR^3$ to
$$(x_2,x_1,x_1+x_2-x_3);\quad (x_1-x_3,-x_2,-x_3);\quad (x_3-x_2,-x_2,x_1-x_2)$$
respectively.
As $S_4$ permutes the generators of $\sigma_{C_4}$, the stable cohomology of $\stratum{}{\sigma_{C_4}}$ is freely generated over the stable cohomology of $\ab$ by four classes of degree $2,4,6,8$ respectively, corresponding to the elementary symmetric functions in the generators of $\sigma_{C_4}$. To identify them as elements of the stable cohomology of $\ua^{\times3}$, we need to extend the generators of $\sigma_{C_4}$ to a basis of $\Sym^2(\RR^3)$ in such a way that the span of the two additional generators $f,g$ is a subrepresentation of $S_4$, e.g.~by setting
\begin{align*}
  f &= -x_1^2+6x_1x_2-x_2^2-2x_1x_3-2x_2x_3+2x_3^2,\\
g &= 2x_1^2 +2x_2^2 -2x_1x_3 -2x_2x_3 -x_3^2.
\end{align*}
Then dualizing gives the following description of the dual elements $\gamma_1,\dots,\gamma_4$ (multiplied by $3$ by convenience) of the generators $\alpha_1,\dots,\alpha_4$ of $\sigma_{C_4}$:
\begin{align*}
\gamma_1&=3\tau_{11}+\tau_{33}+2\tau_{12}+4\tau_{13}+\tau_{23}\mspace{-9mu}
&\mapsto&-6T_1-2T_3-2P_{12}-4P_{13}-P_{23}\\
\gamma_2&=3\tau_{22}+\tau_{33}+2\tau_{12}+\tau_{13}+4\tau_{23}\mspace{-9mu}
&\mapsto&-6T_2-2T_3-2P_{12}-P_{13}-4P_{23}\\
\gamma_3&=\tau_{33}-\tau_{12}-2\tau_{13}+\tau_{23}
&\mapsto&-T_{3}+P_{12}+2P_{13}-P_{23}\\
\gamma_4&=\tau_{33}-\tau_{12}+\tau_{13}-2\tau_{23}
&\mapsto&-T_{3}+P_{12}-P_{13}+2P_{23}\\
\end{align*}

From this it follows that the stable cohomology of $\stratum{}{\sigma_{C_4}}$ is generated by the elementary symmetric functions in the $\gamma_i$. In particular, the degree $2$ generator is
$$\xi'=3T_1+3T_2+4T_3+2P_{23}+2P_{13}+P_{12},\ \text{(degree $2$)}$$
and the degree $4$ generator is
\begin{multline*}
\eta'=-P_{12}^2-P_{12}P_{13}-P_{13}^2-P_{12}P_{23}-P_{23}^2
\\
+12T_1T_2+12T_1T_3+12T_2T_3+6P_{13}T_2+6T_1P_{23}
\\
+4P_{13}P_{23}+4P_{12}T_3+8P_{13}T_3+8P_{23}T_3+8T_3^2.
\end{multline*}

\subsection{Strata of codimension 5}
For the codimension 5 strata the full computation of automorphism groups and of invariant classes becomes more elaborate.
Note, however, that for our purposes we are only interested in the cohomology in degrees up to 2.
Since each of these strata $\stratum{}\sigma$ is connected, the $H^0$ is always one-dimensional, and generated by the Poincar\'e dual of the fundamental class.
By  Theorem \ref{thm:stablestrata}, the $H^1$ vanishes and the $H^2$ is generated by $\lambda_1$ and by classes coming from the $G(\sigma)$-invariant subspace of $\Span(\sigma)$.

The first stratum of codimension 5 corresponds to semiabelic varieties of torus rank 3, and
was also treated in \cite{huto} and in \cite{huto2}, where it is denoted simply by $\sigma^{(5)}$. It is given by
$$\sigma_{K_4-1}=\left<x_1^2,x_2^2,x_3^2,(x_1-x_3)^2,(x_2-x_3)^2\right>,$$
and the full automorphism group was computed in \cite[\S6.5]{huto2}, and it coincides with the subgroup of the automorphism group of $C_4$ fixing $x_3^2$. From this it follows that the $G(\sigma_{K_4-1})$-invariant part of $\Span(\sigma_{K_4-1})$ is two-dimensional, generated by $x_3^2$ and by the sum of the other generators. Dually, this can be viewed inside the stable cohomology of $\ua^{\times3}$  as the span of the two invariants $i_1=T_1+T_2$ and $i_2=4T_3+P_{12}+2P_{13}+2P_{23}$  computed in \cite{huto2}.

Next, there are three strata in torus rank $4$, namely those corresponding to the cones
$$\sigma_{K_3+1+1}=\left<x_1^2,x_2^2,(x_1-x_2)^2,x_3^2,x_4^2\right>,$$
$$\sigma_{C_4+1}=\left<x_1^2,x_2^2,(x_1-x_3)^2,(x_2-x_3)^2,x_4^2\right>$$
and
$$\sigma_{C_5}=\left<x_1^2, x_2^2, (x_1-x_4)^2,(x_2-x_3)^2,(x_3-x_4)^2\right>.$$

In the case of $\sigma_{K_3+1+1}$ the automorphism acts on $\Span(\sigma_{K_3+1+1})$ as the product $S_3\times S_2$, where the first factor permutes the first three generators of $\sigma_{K_3+1+1}$ (as in the case of $\sigma_{K_3}$) and the second factor interchanges the last two generators. Therefore, the invariant subspace of $\Span(\sigma_{K_3+1+1})$ is two-dimensional, generated by $x_1^2+x_2^2+(x_1-x_2)^2$ and $x_3^2+x_4^2$.

For $\sigma_{C_4+1}$ the automorphism group coincides with that of $\sigma_{C_4}$ and acts trivially on $x_4^2$. Therefore, the invariant subspace of $\Span(\sigma_{C_4+1})$ is again two-dimensional, generated by $x_1^2+x_2^2+(x_1-x_3)^2+(x_2-x_3)^2$ and $x_4^2$.

Finally, the automorphism group $G(\sigma_{C_5})$ acts on $\Span(\sigma_{C_5})$ by permuting the $5$ generators. Therefore the invariant part of $\Span(\sigma_{C_5})$ is one-dimensional.

\smallskip
We now encounter a new feature: indeed, as explained in \cite{grhusurvey}, correcting \cite{grhu2}, there
exist two cones in the perfect cone decomposition of codimension 5 and torus rank 5. The first one is the standard cone given by
$$\sigma_{1+1+1+1+1}=\left<x_1^2,x_2^2,x_3^2,x_4^2,x_5^2\right>,$$
for which Lemma \ref{lm:standard} implies that the invariant part of $\Span(\sigma_{1+1+1+1+1})$ is one-dimensional.

The other case corresponds to the {\it non-standard} $5$-dimensional  cone given by
$$\sigma_{NS}:=\left< x_1^2, \ldots x_4^2, (2x_5- x_1 - x_2 - x_3 - x_4)^2\right>.$$
Its reduced automorphism group is generated by the group $S_5$ permutating the five generators of $\sigma_{NS}$. Therefore, the invariant part of $\Span(\sigma_{NS})$ is generated by the sum $x_1^2+\dots+x_4^2+(2x_5- x_1 - x_2 - x_3 - x_4)^2$.

\subsection{Strata of codimension 6}
For the strata of (complex) codimension 6, note that by the Gysin spectral sequence their only cohomology that matters for the computation of $H^{\le 12}(\Perf)$ is the $H^0$. Since each such stratum is connected, its $H^0$ is one-dimensional, and we simply note that there are in total $13$ strata. These correspond to the non-degenerate
$6$-dimensional cones of which there are $1$, $4$, $5$ and $3$ in genus $3$, $4$, $5$ and $6$ respectively, see \cite{numberofperfectforms}.
Note that the $6$-dimensional cones in genus $3$ and $4$ are all matroidal. In genus $5$, four of them --- the cones associated with the graphical lattices $C_6$, $C_5+1$, $C_4+1+1$ and $C_3+1+1+1$ --- are matroidal; the remaining cone contains $\sigma_{NS}$ and is therefore non-matroidal.

From the definition of matroidal cones, it follows that the standard cone (up to the $\GL(g,\ZZ)$-action) is the only $g$-dimensional matroidal cone of rank $g$. Hence, of the three $6$-dimensional perfect cones of genus $6$ one is matroidal and the other two are not.

Finally, we are ready to compute the cohomology of $\Perf$ in degree up to 12.
\begin{proof}[Proof of Theorem \ref{thm:Perfnumbers}]
From the proof of Theorem \ref{thm:main} it follows that for $k<g$, the cohomology of $\Perf$ in degree $\geq\topd-k$  is the direct sum of the stable cohomology with compact support of the strata $\stratum{}\sigma$ of codimension $\leq\lceil g/2\rceil$.
This means that for $g\geq 13$ we can calculate the cohomology of $\Perf$ in degree larger than or equal to $\topd-12$ by collecting the stable Betti numbers calculated in the previous sections and adding them as shown in Table \ref{tab:addingup}. From this the claim follows.
\end{proof}

\begin{table}
\caption{Betti numbers of stable cohomology}

\begin{tabular}{c|ccccccc} \label{tab:addingup}
degree&0&2&4&6&8&10&12\\\hline
$\ab$&1&1&1&2&2&3&4\\
$\beta_1^0$&&1&2&3&5&7&10\\
$\beta_2^0$&&&1&3&6&11&18\\
$\stratum{}{\sigma_{1+1+1}}$&&&&1&2&4&8\\
codim. 4 strata&&&&&3&7&15\\
codim. 5 strata&&&&&&6&15\\
codim. 6 strata&&&&&&&13\\\hline
Tot.&1&2&4&9&18&38&83
\end{tabular}
\end{table}

The cohomology of $\Matr$ in low degree is computed analogously:
\begin{proof}[Proof of Theorem \ref{thm:Matrnumbers}]
To compute the stable cohomology of $\Matr$ we simply need to subtract from Table \ref{tab:addingup} the contribution of the non-matroidal cones and then use Poincar\'e duality to pass from cohomology with compact support in degree $\topd-k$ to cohomology in degree $k$.
The only changes occur in dimensions 10 and 12. In dimension 10 we lose one generator, corresponding to the fundamental class of the
non-standard torus rank 5 codimension 5 cone. In dimension 12, we lose one generator for each of the three non-matroidal cones of dimension 6,
and two  generators for the $H^2$ of the non-matroidal dimension 5 cone.
\end{proof}

\section{Algebraic generators for cohomology}\label{sec:alg}
Above we have computed the dimensions of the stable cohomology groups
$H^{\operatorname{top}-k}(\Perf, \QQ)
\cong H^k(\Perf[g,{\operatorname{smooth}}],\QQ)\cong H^k(\Perf[g,{\operatorname{simp}}],\QQ)$
for $k\le 12$. We will now identify geometrically generators for the cohomology groups for $k\le 8$ and for most of $H^{10}$, and then discuss the phenomena present for $H^{12}$.
To be more precise, we shall construct certain
geometric cycles on the open part $\Perf[g,{\operatorname{simp}}]$
of $\Perf$  where the cycle map $cl: A_{\QQ}^\pu(\Psimp)
\to H^\pu(\Psimp, \QQ)$
is well defined and a ring homomorphism, see \cite[Corollary 19.2]{fultonintersection}.  Naturally this approach also works for the open sets
$\Psmooth$ and $\Matr$ which are (proper) subsets of $\Psimp$.

We will use two methods for constructing cohomology classes, the {\em strata algebra} --- generated by the fundamental classes of the strata in $\Perf$ corresponding to various perfect cone cones --- and the {\em boundary algebra} --- generated by suitable polynomials in irreducible divisorial components of the boundary of the level cover $\Perf(2)$ --- both taken together with the algebra generated by the Hodge classes $\lambda_{2i+1}$.

More precisely, for the first construction, we consider the algebra generated by the fundamental classes of the closures of the strata $\beta(\sigma)$, where $\sigma$ is a simplicial cone --- we call this the strata algebra by analogy with the
subalgebra of the cohomology of the moduli space of curves generated by the fundamental classes of the strata of stable curves of fixed topological type. From now on when we speak about the {\em class of the stratum}, we mean the cohomology class of its closure.
In order to keep the notation manageable, {\em in this section} we will denote  the corresponding cohomology class also by $\sigma$.

The second construction is by going to a level cover $\Perf(2)$, where the boundary becomes a reducible divisor, with its irreducible components $D_m$ labeled by
vectors
 $m\in (\ZZ/2\ZZ)^{2g}\setminus\{0\}$. The boundary components in $\Perf(2)$ corresponding to a basic cone intersect generically transversally.
 By writing polynomials in the classes of $D_m$ invariant under the action of the deck group $\Sp(2g,\ZZ/2\ZZ)$ of the cover $\Perf(2)\to\Perf$
 we obtain classes in suitable open subsets of $\Perf(2)$ (such as the simplicial locus) which descend to $\Perf$.  To avoid unnecessary multiplicities in our notation we
 normalize the pushforward by dividing by the order of the deck group, as was also done in \cite[Section 4]{grhu1}.
This construction provides us with well defined cohomology classes on the simplicial locus $\Perf[g,{\operatorname{simp}}]$.
It was used in \cite{grhu1}, where especially in Sections 8 and 9 similar constructions were performed, and we freely use the notation and results from there. We recall that the intersection of two different
boundary divisors $D_{m_1}\cap D_{m_2}\subset\Perf(2)$ is non-empty if and only if $m_1$ and $m_2$ span an isotropic subspace, i.e.~if and only if the scalar product $m_1\cdot m_2=0\in \ZZ/2\ZZ$. We note that the orbit under $\Sp(2g,\ZZ/2\ZZ)$ of a $k$-tuple $m_1,\ldots,m_k\in(\ZZ/2\ZZ)^{2g}\setminus\{0\}$ such that each pair $m_i,m_j$ is isotropic consists of all $k$-tuples of vectors satisfying the same set of linear relations over $\ZZ/2\ZZ$ (in particular, if some $m_i$ are the same, then in the orbit some of the elements must also be the same). Thus the generators for the vector space of polynomials in $D_m$ invariant under the action of $\Sp(2g,\ZZ/2\ZZ)$ are given by sums of products of the boundary divisors of the form $\sum D_{m_1}\ldots D_{m_k}$ subject to a fixed set of linear relations of the form $m_{i_1}+\ldots+m_{i_\ell}=0$. We will thus proceed by enumerating all such polynomials in $D_m$ of degree up to 6 (calling such polynomial a {\em pure boundary} class),
  and multiplying them by suitable polynomials in the Chern classes $\lambda_{2i+1}$ of the Hodge bundle. To prove that one obtains the entire stable cohomology in a given degree one then has to compare these classes to the ones which we used to prove stability in Sections \ref{sec:stable} and \ref{sec:next} and to compute the explicit numbers in Theorem \ref{thm:Perfnumbers}.

We will see that in degree up to 8 the strata algebra and the boundary algebra are equal, and both are equal to the stable cohomology.
On the other hand in degree 10 neither of them generates the entire stable cohomology $H^{10}(\Perf[g,{\operatorname{smooth}}],\QQ)$, and it appears that they give different codimension 1 subspaces of it. In degree 12 it appears likely that the strata algebra, boundary algebra, $H^{12}(\Perf[g,{\operatorname{smooth}}],\QQ)$, and $H^{12}(\Matr,\QQ)$
are all different.

{\bf Case $k=0$}. Here we of course have one class, which is simply $1$.

{\bf Case $k=2$}.
We have already treated this in Corollary \ref{cor:firststab}.
Here we have one class $\lambda_1$, which already exists on $\ab$,
and one class $\beta_1$, which on the one hand is the closure of the stratum given by the unique rank $1$ cone $\sigma_1$ and on the other
hand  is nothing but the boundary $D$ and can in the spirit of the above
discussion be identified with the sum $\sum D_m$.
Thus we have identified both generators of the stable cohomology $H^2(\Perf)$.

{\bf Case $k=4$}.
In our previous discussion we saw that the stable cohomology in degree $4$ has rank $4$. The only degree 4 class which already lives on $\ab$ is $\lambda_1^2$.
The boundary is the closure of the stratum $\sigma_{1+1}$, which contributes the classes $\lambda_1\beta_1$ and $\beta_1^2$.
Here we note that $\beta_1^2$ equals the class given by $T$ on $\sigma_1$ (and the latter is nothing but the universal abelian variety in genus $g-1$).
Finally the class of the closure of
the stratum $\sigma_{K3}$ also lies in $H^4$. In terms of boundary components of $\Perf(2)$ the first of these classes is $\lambda_1(\sum D_m)$, whereas the last
is given by $\sum_{m\ne m'}D_mD_{m'}$. Finally, $\beta_1^2$ corresponds to  $(\sum D_m)^2= \sum D_m^2  + \sum_{m\ne m'}D_{m}D_{m'}$.

{\bf Case $k=6$}. Here we need to be a bit more methodical. By Theorem \ref{thm:Perfnumbers} the stable $H^6(\Perf[g,\operatorname{smooth}])$ has dimension $9$. There are five classes which are
products of classes of degree at most $4$ (which we have already identified) with $\lambda$-classes. These are the two classes from $\ab$, namely $\lambda_1^3$ and $\lambda_3$, then the degree $2$ classes supported on the boundary multiplied with $\lambda_1^2$, i.e.~$\lambda_1^2\beta_1$, and finally from the degree $4$ classes supported on the boundary we obtain  $\lambda_1\beta_1^2$ and $\lambda_1\beta_2$. So far we have thus constructed five classes that are not obtained as cubic expressions in $D_m$.

We will now enumerate cubic expressions in $D$; this has actually been studied in detail in \cite{grhu1}. However, to set up the more methodical search below, we review how this can be done. Indeed, first of all in a cubic expression some indices $m_i$ may coincide (equivalently, this is a linear relation $m_i+m_j=0$). The expressions where there are some coincidences are thus $\sum D_m^3$ and $\sum D_{m_1}^2 D_{m_2}$, where from now on we use the convention that each such sum is over all possible $m_1,\ldots m_k$ satisfying no additional relations in addition to the ones stated --- so in particular in the second sum $m_1$ and $m_2$ are assumed to be distinct.

If we have a cubic expression with no $m_i$ coinciding, there are actually two cases, corresponding to whether the sum of the three indices is zero or not (these are the so-called local and global, corresponding to whether the three divisors intersect within $\beta_2^0$ or $\beta_3^0$). We thus have the two expressions $\sum_{m_1+m_2+m_3=0} D_{m_1}D_{m_2}D_{m_3}$ and $\sum D_{m_1}D_{m_2}D_{m_3}$ (where recall in the second sum we enforce $m_1+m_2+m_3\ne 0$). Also, from now on, when writing such sums, we will implicitly divide by the suitable product of factorials so that each summand appears only once, that is both of these cubics should be divided by 6, while say $\sum D_{m_1}^2D_{m_2}^2$ would be divided by 4.

Thus we have a total of 4 classes that are cubics in $D_m$. Indeed these four classes, together with the five classes described above generate the stable cohomology in degree $6$.
To see this we note that the condition  $m_1+m_2+m_3=0$ means that the three boundary divisors $D_{m_i}$ intersect locally, i.e.~the generic point of this
intersection is contained in $\beta^0_2$ and this intersection is the closure of the stratum $\beta(\sigma_{K_3})$. On the other hand the condition $m_1+m_2+m_3\ne 0$ means
that the three divisors intersect ``globally'', i.e.~their intersection is contained in $\beta_3$.
In fact this intersection is irreducible and equals $\beta_3$, which in turn is the closure of the stratum
$\sigma_{1+1+1}$ .

For what follows it is useful to use a better formalism for describing homogeneous polynomials in the $D_m$.
To make the formulas readable, we write $\{m_1^{i_1}\ldots m_l^{i_l}\}$ for $\sum D_{m_1}^{i_l}\ldots D_{m_l}^{i_l}$, where we order the powers so that $i_1\ge i_2\ge\ldots\ge i_l$, and furthermore we order the indices so that if  $i_{a-1}>i_a=\ldots=i_b>i_{b+1}$, then $m_a>\ldots>m_b$. We further note the linear relations in parenthesis, so that eg.~$(123)$ means $m_1+m_2+m_3=0$. Thus for example we have $\{1\}=\sum D_m$; $\{12\}=\sum_{i<j} D_i D_j=\beta_2$, and $\{1^22\}=\sum_{i\ne j}D_i^2 D_j$.

In degree $3$ we thus have the four possibilities $\{1^3\}$, $\{1^22\}$, $\{123\}$, $\{123(123)\}$. They relate to the fundamental classes of strata of $\Perf$ as follows:
$\beta_1^3=\{1^3\}+3\{1^22\} +6\{ 123\} + 6\{123(123)\}$, $\beta_1\beta_2=\{1^22\} + 3\{123\} + 3\{123(123)\}$, $\sigma_{K_3}=\{123(123)\}$ and
$\sigma_{1+1+1}=\{123\}$. Hence the space of classes
spanned by the  four possible cubic polynomials in the $D_m$ equals the span of $\beta_1^3$,$\beta_1\beta_2$,$\sigma_{K_3}$, $\sigma_{1+1+1}$.
We note that $\beta_1^3$ is the class $T^2$ on $\beta(\sigma_1)$ and $\beta_1\beta_2$ is the class of $T_1+T_2$ on
$\sigma_{1+1}$, see the proof of Lemma \ref{lm:1+1}.
Thus it follows from
Section \ref{sec:numbers} that the classes obtained as polynomials in the $D_m$ together with the $\lambda$-classes generate the stable cohomology in degree $6$.

{\bf Case $k=8$}. From Theorem \ref{thm:Perfnumbers} we know that the rank of the stable cohomology in degree $8$ is $18$. Above we have described $9$ classes in degree $6$. Multiplying these with $\lambda_1$
and also taking $\lambda_3\beta_1$ we obtain $10$ independent classes. The remaining stable cohomology can be generated by classes which do not contain a factor
which is a $\lambda$-class and,
according to Sections \ref{sec:next} and \ref{sec:numbers}, is generated by the image of the classes
$T^3\sigma_1, (T_1+T_2)^2\sigma_{1+1},  (T_1T_2)\sigma_{1+1}, (T_1+T_2+T_3) \sigma_{1+1+1},
(2(T_1+T_2)-P) \sigma_{K_3}, \sigma_{1+1+1+1},  \sigma_{K_3+1},\sigma_{C_4} $ in $H^{\operatorname{top} - 8}(\Perf,\QQ)$.

We will now show how to obtain (the span of) these classes by the quartic polynomials in $D$.
For this we first have to enumerate these.
This situation was studied in detail in \cite[Proposition 8.4]{grhu1}. The possibilities are
$$
 \{1^4\},\{1^32\},\{1^22^2\},\{1^223(123)\},\{1^223\},
$$
$$
 \{1234(123)\}, \{1234(1234)\},\{1234\},
$$
so that altogether we get 8 classes. As discussed in \cite{grhu1}, their span is equal to the span of the classes $\beta_1^4     ,\beta_1^2\beta_2,\beta_2^2,\beta_1\beta_3,
\beta_1(\sigma_{K_3} + \sigma_{1+1+1}),\beta_4,\{1234\} + \{1234(123)\} + \{1234(1234)\}$, and $\{1234(1234)\}$.
The stratum $\beta_4$  is irreducible, and we have  $\beta_4=\sigma_{1+1+1+1}$, which corresponds
to the polynomial $\{1234\}$. From the definition of the cones $\sigma_{K_3+1}$ and $\sigma_{C_4}$ we find that these strata correspond to $\{1234(123)\}$ and $\{1234(1234)\}$.
Next $\beta_1\beta_3$ gives the class coming from $(T_1+T_2+T_3) \sigma_{1+1+1}$.
Since $\sigma_{1+1+1}=\beta_3$ we obtain, modulo $\beta_1\beta_3$, that $\beta_1(\sigma_{K_3} + \sigma_{1+1+1})$ gives the unique degree $2$
class on $\sigma_{K_3} $ which, by  the proof of Lemma \ref{lm:K3} is $(2(T_1+T_2) - P)\sigma_{K_3}$.
Modulo the classes already enumerated we then see
that $\beta_2^2$, which corresponds to $\{12\}^2$, gives  $(T_1T_2)\sigma_{1+1}$. Similarly $\beta_1^2\beta_2$, which corresponds to
$\{1\}^2\{12\}$, gives, again modulo classes already
enumerated, the class $(T_1+T_2)^2\sigma_{1+1}$. Finally $\beta_1^4$ gives $T^3\sigma_1$ plus classes from above. This shows that we obtain the entire stable cohomology in degree $8$ by using either the strata algebra or the polynomials in $D_m$.

{\bf Case $k=10$}.
Here we will see that neither the boundary algebra nor the strata algebra span all of $H^{10}(\Perf[g,{\operatorname{smooth}}],\QQ)$, while it could be that together they span it.

Indeed, we know from Theorem (\ref{thm:Perfnumbers}) that stable $H^{10}(\Perf[g,{\operatorname{smooth}}],\QQ)=\QQ^{38}$.
Above we have seen that all 19 stable classes of degree 8 lie in the strata algebra and in the boundary algebra. Multiplying each of these 18 classes by $\lambda_1$ gives a degree 10 class in the stable cohomology of $\Perf$. We also have the class $\lambda_5$ in the stable cohomology of $\ab$ (which also extends to $\Perf$). Furthermore, we can construct more classes as a product of $\lambda_3$ and a suitable boundary class. For this, we would need a polynomial in boundary strata of codimension 4, and there are of course two such classes, $\beta_1^2$ and $\beta_2$ (the same space is the linear span of $\stratum{}{\sigma_1}^2$ and $\stratum{}{\sigma_{1+1}}$). Thus altogether we have constructed $21=18+1+2$ degree 10 classes involving a $\lambda$-class.
We thus need to account for the remaining 17 classes in $H^{10}(\Perf[g,\operatorname{smooth}])$.

To understand pure boundary strata in $H^{10}$, we need to study the possible quintics in $D_m$: these are enumerated in the proof of \cite[Proposition 9.1]{grhu1}, and in our notation are as follows:
$$
\{1^5\},\{1^42\},\{1^32^2\},\{1^323\},\{1^323(123)\},\{1^22^23\},\{1^22^23(123)\},
$$
$$\{1^2234\},\{1^2234(1234)\},\{1^2234(123)\},\{1^2234(234)\},\{12345\},
$$
$$
\{12345(12345)\},\{12345(1234)\},\{12345(123)\},\{12345(123,145)\},
$$
which gives  a total of 16 quintic polynomials in $D_m$. Thus the dimension of pure boundary algebra is 16, and together with the 21 classes enumerated above these are insufficient to generate the stable $H^{10}(\Perf[g,{\operatorname{smooth}}],\QQ)$. Thus the boundary algebra is {\em smaller} than $H^{10}(\Perf[g,{\operatorname{smooth}}],\QQ)$.

Similarly, for the strata algebra in degree 10, as per the discussion in Section (\ref{sec:numbers}), we note that there are 6 boundary strata of complex
codimension 5. These can be related to polynomials in the $D_m$  as follows:
$$
 \sigma_{1+1+1+1+1}=\{12345\},\qquad  \sigma_{K_3+1+1}=\{12345(123)\}
$$
$$
  \sigma_{C_4+1}=\{12345(1234)\},\qquad  \sigma_{K_4-1}=\{12345(123,145)\}
$$
while we have
$$
  \sigma_{C_5}+\sigma_{NS}=\{12345(12345)\},
$$
where we recall that $\sigma_{NS}$ denotes the non-standard non-matroidal cone. The last identity follows since the 5-tuples given by the generators of
the cones  $\sigma_{C_5}$ and $\sigma_{NS}$ coincide mod $2$.
Thus all 5 quintics in $D_m$ that involve 5 different indices can be expressed in terms of boundary strata, but not vice versa.
We now investigate further degree 10 classes in the strata algebra. For polynomials that involve a boundary class of complex codimension 4, we have
$$\begin{aligned}
{\sigma_1} \sigma_{1+1+1+1}&= *\{1^2234\}+*\{12345\}+*\{12345(12345)\}\\
  {\sigma_1} \sigma_{K_3+1}&= *\{1^2234(123)\}+*\{1^2234(234)\}+*\{12345(123)\}\\
  & \quad+ *\{12345(123,145)\}\\
   {\sigma_1} \sigma_{C_4+1}&=*\{1^234(1234)\}+*\{12345(1234)\}
\end{aligned}
$$
where $*$ denotes the various combinatorial non-zero coefficients appearing. Using these expressions, together with the expressions for the quintics involving five different $D_m$, obtained above, we can express $\{1^2234\}$ and $\{1^234(1234)\}$ as linear combinations of polynomials in boundary strata. Note, however, that so far we are only able to express a suitable linear combination $*\{1^2234(123)\}+*\{1^2234(234)\}$ as a polynomial in boundary strata --- but not the two summands individually.

For the elements of the (pure, not involving the $\lambda$'s) strata algebra involving a cone of  complex codimension 3, we similarly have
$$
\begin{aligned}
 \sigma_1^2\sigma_{1+1+1}&=*\{1^323\}+*\{1^22^23\}+X\\
 \sigma_1^2 \sigma_{K3}&=*\{1^323(123)\}+*\{1^22^23(123)\}+X\\
 {\sigma_{1+1}}{\sigma_{1+1+1}}&=*\{1^22^23\}+X\\
 {\sigma_{1+1}}{\sigma_{K3}}&=*\{1^22^23(123)\}+X
\end{aligned}
$$
where $X$ in each case denotes various explicit linear combinations of quintics involving at least 4 different $D_m$. Thus from the above expressions, we can express each of the four quintic polynomials $\{1^323\}$, $\{1^323(123)\}$, $\{1^22^23\}$, $\{1^22^23(123)\}$, as a linear combination of monomials in boundary strata and quintics involving at least 4 different $D_m$, while we get no further information or relations that could allow us to distinguish $*\{1^2234(123)\}$ and $*\{1^2234(234)\}$ (or $\sigma_{NS}$ and $\sigma_{C_5}$).

It remains to enumerate elements of the strata algebra that only involve classes of codimension at most 2; that is to say, we now need to write monomials in $\sigma_1$ and
$\stratum{}{\sigma_{1+1}}$ only. Again, now denoting $X$ any linear combinations of quintics involving at least 3 different $D_m$, we get
$$
\begin{aligned}
 \sigma_1^5&=*\{1^5\}+*\{1^42\}+*\{1^32^2\}+X\\
 \sigma_1^3 {\sigma_{1+1}}&=*\{1^42\}+*\{1^32^2\}+X\\
 {\sigma_1} \sigma_{1+1}^2&=*\{1^32^2\}+X
\end{aligned}
$$
so that again these monomials can be expressed in terms of the polynomials in the strata algebra and the monomials we have studied previously. We thus obtain

\begin{summary} There exist 16 pure boundary classes (quintics in $D_m$) and 16 pure strata classes (monomials in the classes of the strata), such that:
\begin{itemize}
\item[(i)] each pure boundary class except $\{1^2234(123)\}$ and $\{1^2234(234)\}$ lies in the pure strata algebra; moreover, a suitable linear combination $*\{1^2234(123)\}+*\{1^2234(234)\}$ lies in the pure strata algebra
\item[(ii)] each pure strata class, except $\sigma_{C_5}$ and $\sigma_{NS}$ lies in the pure boundary algebra; moreover, $\sigma_{C_5}+\sigma_{NS}$ also lies in the pure boundary algebra (in fact, is equal simply to $\{12345(12345)\}$)
\end{itemize}
\end{summary}
We thus obtain
\begin{prop} \label{prop:notenough}
Neither the strata algebra nor the boundary algebra generate the cohomology rings of either the smooth or the simplicial locus of $\Perf$.
\end{prop}
We furthermore {\em conjecture} that in fact both the boundary algebra and the strata algebra in degree 10 have dimension 37, that together they span $H^{10}(\Perf[g,{\operatorname{smooth}}],\QQ)$, and moreover that the boundary algebra actually is equal to $H^{10}(\Matr)$.
As the stratum $\sigma_{NS}$ does not belong to $\Matr$, the dimension of the restriction of the strata algebra to $\Matr$ is only $36$ in degree $10$, so it is clear that the strata algebra cannot give all stable cohomology of $\Matr$.

{\bf Case $k=12$}.
Here we will see that $H^{12}(\Perf[g,{\operatorname{smooth}}],\QQ)$, $H^{12}(\Matr,\QQ)$, the boundary and the strata algebra {\em all seem to have different dimensions}.
We recall that by Theorem \ref{thm:Perfnumbers} and Theorem \ref{thm:Matrnumbers} we have
$H^{12}(\Perf[g,{\operatorname{smooth}}],\QQ)=\QQ^{83}$ and $H^{12}(\Matr,\QQ)=\QQ^{78}$.

For both the boundary and the strata algebra, we shall first enumerate those classes which
involve $\lambda$-factors.
Here we have 4 classes in the interior: $\lambda_1^6,\lambda_1^3\lambda_3,\lambda_1\lambda_5$, $3\cdot 1$ classes by multiplying the $k=10$ interior classes by the unique pure boundary class $\lbrace 1\rbrace$ (which is the same as the pure stratum class $\sigma_1$, of course),
$2\cdot 2$ classes by multiplying the $k=8$ interior classes by the two pure boundary/strata classes in degree 4, $2\cdot 4$ classes by multiplying the $k=6$ interior classes by the two pure boundary/strata classes in degree 6, $1\cdot 8$ classes by multiplying the $k=4$ interior class $\lambda_1^2$ by the pure boundary/strata classes in degree 8, and $1\cdot 16$ classes by multiplying $\lambda_1$ by the quintics in $D_m$ (which are different from the 16 pure strata classes, but the dimension is the same), for a total of $4+3+4+8+8+16=43$ classes.

Next we discuss the pure boundary classes, i.e.~sextic polynomials in  $D_m$. We have
$$\{1^6\},\{1^52\},\{1^42^2\},\{1^32^3\},\{1^423\},\{1^423(123)\},\{1^32^23\},\{1^32^23(123)\},$$
$$\{1^22^23^2\},\{1^22^23^2(123)\},\{1^3234\},\{1^3234(1234)\},\{1^3234(123)\},$$
$$\{1^3234(234)\},\{1^22^234\},\{1^22^234(1234)\},\{1^22^234(123)\},\{1^22^234(134)\},  $$
$$\{1^22345\},\{1^22345(12345)\},\{1^22345(1234)\},\{1^22345(2345)\},$$ $$\{1^22345(123)\},\{1^22345(234)\},\{1^22345(123,145)\},\{1^22345(123,245)\},  $$
$$\{123456\},\{123456(123456)\},\{123456(12345)\},\{123456(1234)\},$$
$$\{123456(1234,1256)\},\{123456(1234,156)\},\{123456(123)\},$$
$$\{123456(123,145)\},\{123456(123,145,246)\},\{123456(123,456)\}$$
for a total of 36 sextics, so that the total dimension of the boundary algebra in degree 12 is at most $43+36=79$ (it could be less as we have not ruled linear relations among the above, which, however, seem unlikely to exist).

We will now discuss the pure strata classes, and will take this opportunity to set up this approach more systematically. We first list the cones of the perfect classes of boundary strata, in each codimension:

\smallskip

\begin{tabular}{|r|c|r|}
\hline
Codim&Cones&\# of Cones\\
\hline
2&$\sigma_1$&1\\
4&$\sigma_{1+1}$&1\\
6&$\sigma_{K_3},\sigma_{1+1+1}$&2\\
8&$\sigma_{K_3+1},\sigma_{C_4},\sigma_{1+1+1+1}$&3\\
10&$\sigma_{K_4-1},\sigma_{K_3+1+1},\sigma_{C_4+1},\sigma_{C_5},\sigma_{1+1+1+1+1},\sigma_{NS}$&6\\
12&\dots&13\\
\hline
\end{tabular}

\smallskip
To compute the number of pure strata classes in degree $k$ is to compute the number of monomials in the classes of these cones, of appropriate degree. Thus we need to sum over all partitions $k=2n_1+\ldots+2n_i$ with the products of the numbers of cones in codimension $2n_i$, from the table above. We have of course implicitly used this throughout the computations above, but there we also were able to identify the individual monomials with the stable cohomology generators or with the boundary algebra. Here we only do the combinatorics; the result is given by the following table, where the results for degree up to 10 simply summarize the previous discussion, and the number of pure strata classes in degree 12 is what we wanted.

\smallskip

\begin{tabular}{|r|r|r|r|}
\hline
&&\\
$k$&partitions of $k$&\# of pure strata classes in $\deg k$\\
&&\\
\hline
2&1&1\\
4&2,11&$1+1\cdot 1=2$\\
6&3,21,111&$2+1\cdot 1+1\cdot 1\cdot 1=4$\\
8&4,31,22,211,1111&$3+2\cdot 1+1+1+1=8$\\
10&5,41,32,311,221,2111,11111&$6+3+2+2+1+1+1=16$\\
12&6,51,42,411,33,321,3111,&$13+6+3+3+2\cdot 2+2+2\ \ $\\
&222,2211,21111,111111&$+1+1+1+1=37$\\
\hline
\end{tabular}

\smallskip
\begin{summary}
We thus have
\begin{itemize}
\item[(i)] The dimension of the strata algebra in degree 12 is equal to at most 80, of which at most 37 are the pure strata classes.
\item[(ii)] The dimension of the boundary algebra in degree 12 is equal to at most 79, of which at most 36 are the pure boundary classes.
\end{itemize}
\end{summary}
It thus follows that neither the strata nor the boundary algebra in degree 12 generate all of $H^{12}(\Perf[g,{\operatorname{smooth}}],\QQ)$ --- which of course is not surprising given that this fails already in degree 10. We would like to close with the following:
\begin{conj}
There are no stable relations in the strata or boundary algebra. More precisely, the strata and boundary algebra are freely generated by the odd lambda classes and the strata, respectively boundary classes for $k<<g$.  
\end{conj}
\begin{qu}
Is it true that the strata and the boundary algebras together generate the stable cohomology of $\Perf[g,\operatorname{smooth}]$?
\end{qu}
\begin{qu}
What is (stably) the intersection of the strata and the boundary algebra?
\end{qu}
\begin{qu}
Is it true that the boundary algebra generates the stable cohomology of $\Matr$?
\end{qu}

We hope that we, or others, would be able to address some of these questions in the future.

\newcommand{\etalchar}[1]{$^{#1}$}

\end{document}